\RequirePackage[english]{babel} 

\documentclass[production,11pt]{jsg}
\usepackage[utf8]{inputenc}      % \PrerenderUnicode{ç}
\usepackage[T1]{fontenc}         % enables hyphenation of accented words
\usepackage{amsmath}
\usepackage{amssymb}
\usepackage{amsthm}  \swapnumbers
\usepackage{textcomp}
\usepackage{graphicx}
\usepackage{epigraph}
\usepackage{paralist}
\usepackage[dvipsnames]{xcolor}  % load before tikz to get the option
\usepackage{tikz}

\makeatletter
  \def\swappedhead@plain#1#2#3{%
    \thmnumber{(\textup{#2})}%                    % upright number
    \thmname{\@ifnotempty{#2}{~}\textup{#1}}%     % upright name
    \thmnote{ {\textup{(#3)}}}}%                  % upright note
  \let\swappedhead\swappedhead@plain
\makeatother

\theoremstyle{plain}
\newtheorem{theo}[equation]{Theorem}%[section]
\newtheorem{prop}[equation]{Proposition}
\newtheorem{lemm}[equation]{Lemma}
\newtheorem{coro}[equation]{Corollary}
\newtheorem{coro1}[equation]{Corollary 1}
\newtheorem{coro2}[equation]{Corollary 2}
\newtheorem{coro3}[equation]{Corollary 3}
\theoremstyle{definition}
\newtheorem{defi}[equation]{Definition}
\newtheorem{exem}[equation]{Example}

\newtheorem{exemSSD}[equation]{Souriau's theorems}
\newtheorem{exemNS}[equation]{A primary orbit which doesn't split}
\newtheorem{exemNTS}[equation]{A primary orbit which splits nontrivially}
\newtheorem{exemTNS}[equation]{A primary space which doesn't split even topologically}
\newtheorem{rema}[equation]{Remark}
\newtheorem{remas}[equation]{Remarks}
\numberwithin{equation}{section}

\usepackage{Z-tikz}

\DeclareMathSymbol{A          }{\mathalpha}{operators}{`A}
\DeclareMathSymbol{B          }{\mathalpha}{operators}{`B}
\DeclareMathSymbol{C          }{\mathalpha}{operators}{`C}
\DeclareMathSymbol{D          }{\mathalpha}{operators}{`D}
\DeclareMathSymbol{E          }{\mathalpha}{operators}{`E}
\DeclareMathSymbol{F          }{\mathalpha}{operators}{`F}
\DeclareMathSymbol{G          }{\mathalpha}{operators}{`G}
\DeclareMathSymbol{H          }{\mathalpha}{operators}{`H}
\DeclareMathSymbol{I          }{\mathalpha}{operators}{`I}
\DeclareMathSymbol{J          }{\mathalpha}{operators}{`J}
\DeclareMathSymbol{K          }{\mathalpha}{operators}{`K}
\DeclareMathSymbol{L          }{\mathalpha}{operators}{`L}
\DeclareMathSymbol{M          }{\mathalpha}{operators}{`M}
\DeclareMathSymbol{N          }{\mathalpha}{operators}{`N}
\DeclareMathSymbol{O          }{\mathalpha}{operators}{`O}
\DeclareMathSymbol{P          }{\mathalpha}{operators}{`P}
\DeclareMathSymbol{Q          }{\mathalpha}{operators}{`Q}
\DeclareMathSymbol{R          }{\mathalpha}{operators}{`R}
\DeclareMathSymbol{S          }{\mathalpha}{operators}{`S}
\DeclareMathSymbol{T          }{\mathalpha}{operators}{`T}
\DeclareMathSymbol{U          }{\mathalpha}{operators}{`U}
\DeclareMathSymbol{V          }{\mathalpha}{operators}{`V}
\DeclareMathSymbol{W          }{\mathalpha}{operators}{`W}
\DeclareMathSymbol{X          }{\mathalpha}{operators}{`X}
\DeclareMathSymbol{Y          }{\mathalpha}{operators}{`Y}
\DeclareMathSymbol{Z          }{\mathalpha}{operators}{`Z}

\usepackage{bm}     % bold math, must come *after* all font declarations

\usepackage[backend=biber,bibstyle=Z-bib,citestyle=alphabetic]{biblatex}

\usepackage[colorlinks=true,%
            linkcolor=RoyalBlue,%
            citecolor=RedViolet,%
            urlcolor=RoyalBlue]{hyperref}   % hyperref, must come *last*

\let\polishl\l

  \newcommand\CC{{\mathbf{C}}}         % the complex numbers
  \newcommand\EFG{{\Gamma}}            % equivariant fundamental group
  \newcommand\G{{G}}                   % ambient group
  \newcommand\Gcov{{\smash{\cov\G}}}   % covering of \G with group \EFG
  \newcommand\GU{{\G^1}}                % stabilizer of \U in \G
  \newcommand\K{K}                     % stabilizer \N_\a
  \newcommand\Ko{{\K\o}}               % identity component of \N\sub\a\o
\renewcommand\L{L}                     % stabilizer \G_\a
  \newcommand\Lcov{{\smash{\cov\L}}}   % copy of \L inside \Gcov
  \newcommand\N{{N}}                   % normal subgroup
  \newcommand\Ncov{{\smash{\cov\N}}}   % copy of \N inside \Gcov
  \newcommand\RR{{\mathbf{R}}}         % the reals
  \newcommand\TT{{\mathbf{T}}}         % the 1-torus
  \newcommand\ZZ{{\mathbf{Z}}}         % the integers

  \newcommand\efg{{\gamma}}            % element of \EFG
  \newcommand\g{{g}}                   % element of \G
  \newcommand\gcov{{\cov\g}}           % element of \Gcov
  \newcommand\h{{l}}                   % element of \H
\renewcommand\k{{k}}                   % element of \K
\renewcommand\l{{l}}                   % element of \L
  \newcommand\n{{n}}                   % element of \N
\renewcommand\o{^{\smash{\mathrm o}}}  % identity component

  \newcommand\Lie{\mathfrak}
  \newcommand\LG{{\Lie{g}}}            % Lie algebra of \G
  \newcommand\LJ{{\Lie{j}}}            % kernel of \a_{|\LK}
  \newcommand\LK{{\Lie{k}}}            % stabilizer \LN_\a
  \newcommand\LL{{\Lie{l}}}            % stabilizer \LG_\a               
  \newcommand\LN{{\Lie{n}}}            % Lie algebra of \N

  \newcommand\lie{\greektt}
\renewcommand\lg{{\lie{\gamma}}}       % element of \LG
\renewcommand\ll{{\lie{\lambda}}}      % element of \LL
\renewcommand\ln{{\lie{\nu}}}          % element of \LN

  \newcommand\T{{T}}                   % trivial hamiltonian \N-space
  \newcommand\U{{U}}                   % orbital motions
  \newcommand\Ucov{{\smash{\cov\U}}}   % maximal covering of \U
  \newcommand\X{{X}}                   % hamiltonian|primary \G|\N-space
  \newcommand\XU{{\X^1}}                % reduced space of \X at \U
  \newcommand\Y{{Y}}                   % other hamiltonian \N-space
  \newcommand\Z{{V}}                   % fiber \p\inv(\a)
  \newcommand\Zdot{{\smash{\dot\Z}}}   % \Z/\EFG

\renewcommand\a{{c}}                   % base point of \U
  \newcommand\acheck{{\check\a}}       % extends \a to \LG
\renewcommand\r{{r}}                   % 3rd Darboux coordinate
  \newcommand\s{{s}}                   % 4th Darboux coordinate 
\renewcommand\u{{u}}                   % element of \U
  \newcommand\ucov{{\cov\u}}           % element of \Ucov
  \newcommand\x{{x}}                   % element of \X
  \newcommand\y{{y}}                   % element of \Y
  \newcommand\z{{v}}                   % element of \Z
  \newcommand\w{\omega}                % 2-form
  \newcommand\wU{\w_\U}                % 2-form on \U      
  \newcommand\wUcov{\w_\Ucov}          % 2-form on \Ucov
  \newcommand\wY{{\sigma}}             % other 2-form
  \newcommand\wZ{\w_\Z}                % 2-form on \Z          

  \newcommand\AD[1]{{\underline{#1}}}  % conjugation by #1
  \DeclareMathOperator\Ad{Ad}          % adjoint representation
  \newcommand\antidiag[1]{{\Delta(#1)}}% k --> (k\inv,k)   
  \newcommand\e[1]{{\mathrm e^{#1}}}   % exponential map
  
  \newcommand\inv{^{-1}}                % inverse
  \newcommand\PHI{{\Phi}}               % moment map \X --> \LG*
  \newcommand\PI{{\Pi}}                 % moment map \X --> \LN*
  \newcommand\PIY{{\Psi}}               % moment map \Y --> \LN*
  \newcommand\PSI{{\Psi}}               % moment map \Z --> \LL*
  \newcommand\pGLK{{p}}                 % projection \LG\to\LL/\LK
  \newcommand\proj{{\pi}}               % projection \N x \L --> \G
  \newcommand\pU{{\rho}}                % covering map \Ucov --> \U
  \newcommand\pX{{\alpha}}              % \alpha
  \newcommand\pY{{\beta}}               % \beta
  \newcommand\phiU{\phi}                % moment map \Ucov --> \LG*  \Phi_\Ucov
  \newcommand\phiZ{\psi}                % moment map \Z --> \LG*     \Phi_\Z
  \let\Re\relax
  \DeclareMathOperator\Re{Re}           % real part

  \newcommand{\<}{\langle}
\renewcommand{\>}{\rangle}
  \DeclareFontFamily{OMX}{MnSymbolE}{}  
  \DeclareFontShape{OMX}{MnSymbolE}{m}{n}{
      <-6>  [0.85] MnSymbolE5
     <6-7>  [0.85] MnSymbolE6
     <7-8>  [0.85] MnSymbolE7
     <8-9>  [0.85] MnSymbolE-Bold8
     <9-10> [0.85] MnSymbolE-Bold9
    <10-12> [0.85] MnSymbolE-Bold10
    <12->   [0.85] MnSymbolE12}{}
  \DeclareMathOperator\ann{ann}
  \newcommand\bracl{{\text{\raisebox{.15ex}{\usefont{OMX}{MnSymbolE}{m}{n}B}}}}
  \newcommand\bracr{{\text{\raisebox{.15ex}{\usefont{OMX}{MnSymbolE}{m}{n}G}}}}
  \newcommand\cov{\tilde}
\renewcommand\d{{\delta}}

  \newcommand\I{{\iota}}
  \newcommand\Ind{{\operatorname{Ind}}}
  \newcommand\INT{\AD}
  
\renewcommand\j{{\hspace{.06em}j\hspace{-.06em}}}                   % 2\pi i
  \newcommand\Irr[1]{{\mathrm{Irr}(#1)}}
  \newcommand\moment{moment }
  \newcommand\sub[1]{_{\smash{\raisebox{1pt}{$\scriptstyle #1$}}}}

\providecommand\greektt{}

\showhyphens{Heisenberg hamiltonian}

\showhyphens{decomposition diffeomorphism equivariance imprimitivity isomorphism Kazhdan modules momentum representation represents represented submersion submersions theoretic}

\bibliography{biblio}
\pgfrealjobname{primary}

\title[Primary Spaces]{Primary Spaces, \,Mackey's Obstruction, \,and the Generalized Barycentric Decomposition}

\author{Patrick Iglesias-Zemmour}
\email{piz@math.huji.ac.il}
\address{Laboratoire d'Analyse, Topologie et Probabilités/CNRS, Aix-Marseille Université, F\hspace{1pt}-13453 Marseille, France}

\author{François Ziegler}
\email{fziegler@georgiasouthern.edu}
\address{Department of Mathematical Sciences, Georgia Southern University, Statesboro, GA 30460-8093, USA}

\date{March 25, 2012}
\pubyr{0000} %%% Publishing Year
\vol{0} %%%% Volume number
\iss{0} %%%% Issue number
\startpage{1} %%% Starting pagenumber
\endpage{23} %%% End pagenumber
\thanks{Received 03/26/2012, accepted 10/30/2013}

\begin{document}

\begin{abstract}
   We call a hamiltonian N-space \emph{primary} if its \moment map is onto a single coadjoint orbit. The question has long been open whether such spaces always split as (homogeneous)\,$\times$\,(trivial), as an analogy with representation theory might suggest. For instance, Souriau's \emph{barycentric decomposition theorem} asserts just this when N is a Heisenberg group. For general N, we give explicit examples which do not split, and show instead that primary spaces are always flat bundles over the coadjoint orbit. This provides the missing piece for a full ``Mackey theory'' of hamiltonian G\nobreakdash-spaces, where G is an overgroup in which N is normal.
\end{abstract}

\maketitle

% see http://article.gmane.org/gmane.comp.tex.tetex.general/688 :
% \let\languagename\relax   
% \setcounter{tocdepth}{1}
% \tableofcontents

\setlength\epigraphwidth{24.4em} 
\renewcommand{\epigraphsize}{\footnotesize}
\setlength\epigraphrule{0pt}
\epigraph{Il est bien entendu que chaque moment, confondu qu'il est en tous les autres, demeure pourtant en lui-même différencié.}{---André Breton, \emph{Les vases communicants}}

%%%%%%%%%%%%%%%%%%%%%%%%%%%%%%%%%%%%%%%%%%%%%%%%%%%%%%%%%%%%%%%%%%%%%%%%
\section*{Introduction}
%%%%%%%%%%%%%%%%%%%%%%%%%%%%%%%%%%%%%%%%%%%%%%%%%%%%%%%%%%%%%%%%%%%%%%%%

Let $\G$ be a finite group and let $\N$ be a normal subgroup, so that we have an exact sequence
\begin{equation*}
  1\longrightarrow
  \N\longrightarrow
  \G\longrightarrow
  \G/\N\longrightarrow
  1.
\end{equation*}
Questions about $\G$ then often reduce to similar ones about $\N$ and $\G/\N$. A classic example is the calculation of the dual $\Irr\G=\{$irreducible unitary $\G$\nobreakdash-modules$\}/$isomorphism. This, as Clifford \cite{Clifford:1937a} showed, boils down in 3 steps to finding $\Irr\N$ and the projective duals of subgroups of $\G/\N$:
\begin{enumerate}[(1)]
  \setcounter{equation}{3}
  \item \label{step1}
  When restricted to $\N$, every  $\X\in\Irr\G$ decomposes into the irreducibles belonging to a single $\G$-orbit $\G(\U)\subset\Irr\N$. Here the action of $\G$ on $\Irr\N$ is contragredient to its conjugation action on $\N$.
   
  \item \label{step2}
  Write $\G(\U)=\G/\GU$. Then we have $\X=\Ind_{\GU}^\G\,\XU$ for some~\mbox{$\XU\in\Irr{\GU}$} whose restriction to $\N$ is a multiple of $\U$. Leaving out the superscripts, we are reduced to the \emph{primary case} where $\U$ is $\G$-invariant and $\X$ itself restricts to a multiple of $\U$. 

  \item \label{step3}
  In that case we have a factorization $\X=\U\otimes\T$, where the action of $\G$ on each factor is only projective. On $\U$, it uniquely extends the linear action of $\N$. On $\T$, it comes from a projective action of $\G/\N$. Both give rise to central extensions of $\G/\N$ by the circle, and these add up to zero in $\mathrm{H}^2(\G/\N,S^1)$.

\end{enumerate}
In \cite{Mackey:1958} Mackey famously extended this analysis to locally compact groups, and the central extension attached to $\smash{\U\in\Irr\N^\G}$ has since been known as the \emph{Mackey obstruction of} $\U$. (For an exposition see \cite[pp.\,1263, 1285]{Fell:1988b}.) 

Now assume that $\G$, $\N$, etc.~are Lie groups with Lie algebras $\LG$, $\LN$, etc. In an agenda pioneered by Kirillov, Kostant, and Souriau, one expects propositions about (not only irreducible) unitary $\G$\nobreakdash-modules to be reflected by propositions about hamiltonian $\G$\nobreakdash-spaces (i.e., symplectic manifolds on which our groups act with equivariant \moment maps). Thus for instance, one expects that
\begin{itemize}[--]
   \item the role of decomposition into irreducibles is played by the decomposition of the image of the \moment map into coadjoint orbits \cite{Kirillov:1962}; 
   \item the role of multiplicities (such as $\T$ above) is played by Marsden-Weinstein reduced spaces \cite{Marsden:1974,Guillemin:1982};
   \item the role of tensor product is played by the cartesian product of hamiltonian $\G$-spaces \cite[p.\,32]{Weinstein:1977a};
   \item the role of unitary induction is played by the symplectic induction construction of Kazhdan-Kostant-Sternberg \cite[p.\,498]{Kazhdan:1978};
   \item the role of Mackey's imprimitivity theorem is played by the symplectic imprimitivity theorem of \cites[Thm 2.9]{Ziegler:1996a}[Thm 4]{Landsman:2006}; etc.
\end{itemize}
Using these ingredients, a complete geometric parallel of steps (\ref{step1}) and (\ref{step2}) was established  in \cite[Prop 3.8]{Ziegler:1996a}. This reduces the study of a hamiltonian $\G$-space $\X$, when $\G$ has a closed normal subgroup $\N$, to the \emph{primary case} where the \moment map for $\N$ is onto a single coadjoint orbit $\U$. (We can think of this case as lying at the opposite end of the spectrum from the \emph{multiplicity-free case} of \cite{Guillemin:1984a}, where all reduced spaces are points \cite[{}45.8]{Guillemin:1984}.)

The purpose of this paper is to unravel the structure of primary spaces and thereby to complete step (\ref{step3}), which has proved more elusive. Here the representation theoretic analogy might lead one to expect a splitting $\X=\bar\U\times\T$ where $\bar\U$ is homogeneous (a covering of $\U$) and $\T$ a trivial  $\N$\nobreakdash-space. The earliest results of this kind are those of Souriau \cite{Souriau:1970}, which can be viewed as asserting just such a splitting when $\N$ is a Heisenberg group; see §4.1 below. The case $\N=\mathrm{SO}(3)$ was dealt with in \cites{Iglesias:1984}[Thm 4.3]{Iglesias:1991}, and further generalizations occur in \cites{Lisiecki:1992}{Li:1994}[§4]{Ziegler:1996a}, but always under hypotheses that guarantee a splitting of the above form. To our knowledge, the literature contains neither a theory nor in fact a single example of the non-split situation, which we will see does actually occur.

The paper is divided into four sections. In section 1 we disregard the role of any ambient group $\G$ (or, what amounts to the same, assume that $\G=\N$) and show that while perhaps not split, primary $\N$\nobreakdash-spaces over a given coadjoint orbit $\U$ are always \emph{flat bundles} $\Ucov\times_\EFG\Z$ associated to a certain covering of $\U$. (As Alan Weinstein kindly points out, this is a general property of Poisson maps to symplectic manifolds: \cite[§7.6]{Cannas-da-Silva:1999}.) Section 2 shows how the extra structure arising from the hamiltonian action of a group $\G$ normalizing $\N$ is encoded in the fiber $\Z$. Section 3 then spells out three corollaries which generalize Souriau's theorems to an arbitrary Lie group $\N$, and comprise our symplectic ``translation'' of the Mackey obstruction step (\ref{step3}). Finally Section 4 shows that our theory is not empty by exhibiting examples which indeed do not split.

We conclude this introduction with two remarks of principle. First, while our notion of primary space arises naturally from the geometrical problem we set out to solve, one should not overestimate its parallelism with primary representations. Indeed several representations (parametrized by the characters of $\EFG$) may correspond to an orbit $\U$, and our non-split examples seem to correspond to the situation where one representation of $\G$ decomposes into many of them. Secondly, we have scrupulously avoided any connectedness assumptions, as these would spoil the recursive applicability of our symplectic ``Mackey machine''; this explains why all coverings are constructed by hand, since one knows that the universal covering of a disconnected group or homogeneous space is not in general a group or homogeneous space \cite{Taylor:1954}.

%%%%%%%%%%%%%%%%%%%%%%%%%%%%%%%%%%%%%%%%%%%%%%%%%%%%%%%%%%%%%%%%%%%%%%%%
\section{Decomposition of a primary $\N$-space}
%%%%%%%%%%%%%%%%%%%%%%%%%%%%%%%%%%%%%%%%%%%%%%%%%%%%%%%%%%%%%%%%%%%%%%%%
\label{section1}

Let $\N$ be a Lie group and let $(\X,\w,\PI)$ be a hamiltonian $\N$-space, that is, a symplectic manifold $\X$ with an action of $\N$ which preserves its 2-form $\w$, and an equivariant \moment map $\PI:\X\to\LN^*$. If $\w$ or $\PI$ are understood, we may drop them from the notation or use subscripts such as $\w_\X$, $\PI_\X$, etc.

\begin{defi}
   We say that $\X$ is \emph{\textbf{primary}} if the image of the \moment map is a single coadjoint orbit $\U$ of $\N$. When $\N$ (or $\U$) needs emphasis we say \emph{$\N$\textbf{\nobreakdash-primary}} (\emph{\textbf{over} $\U$}). 
\end{defi}

Our purpose is to describe and classify primary $\N$-spaces over a given orbit $\U$. We do not distinguish between two spaces $\X_1$, $\X_2$ which are \emph{isomorphic}, i.e., related by an $\N$-equivariant diffeomorphism which transforms $\w_1$ into $\w_2$ and $\PI_1$ into $\PI_2$. 

\begin{exem}[Homogeneous spaces; split spaces]
   Any \emph{homogeneous} hamiltonian $\N$-space is primary: indeed one knows that its \moment map is a covering map, $\bar\U\to\U$, onto a coadjoint orbit \cite{Kostant:1970,Souriau:1970}. More generally it is clear that the product
   \begin{equation}
      \label{split}
      \bar\U\times\T
   \end{equation}
   of such a covering by any trivial hamiltonian $\N$-space, is primary over $\U$. Here by a \emph{trivial} hamiltonian $\N$-space $\T$ we mean one with trivial $\N$-action and zero \moment map, $\T\to\{0\}$; and we recall that the \emph{product} $\X\times\Y$ of two hamiltonian $\N$-spaces is the product manifold with 2-form $\w_\X+\w_\Y$, action $\n(\x,\y)=(\n(\x),n(\y))$, and \moment map $\PI_\X+\PI_\Y$.
\end{exem}

\begin{defi}
   \label{splitDefi}
   We say that a primary space over $\U$ \emph{\textbf{splits}} if it has the form \eqref{split} ($\bar\U$ homogeneous, $\T$ trivial). We further characterize the splitting as \emph{\textbf{trivial}} or not according as the covering $\bar\U\to\U$ is trivial (bijective)~or~not. 
\end{defi}

\begin{exem}[Flat bundles]
   \label{flatBundles}
   Split spaces, it will turn out, do not exhaust primary spaces. Instead we have to resort to the next simplest construction, which we now present. First observe that the class of homogeneous $\N$-spaces over $\U$ contains a maximal object, $\Ucov$, which we may describe as follows. Write $\K$ for the stabilizer $\N_\a$ of a point $\a\in\U$ which we fix once and for all, $\Ko$ for its identity component, and 
\begin{equation}
   \label{Ucov}
   \Ucov\overset{\pU}{\longrightarrow}\U
\end{equation}
for the covering $\N/\Ko\longrightarrow\N/\K$ with principal group
\begin{equation}
   \label{EFG}
   \EFG = \K/\Ko
\end{equation}
(a.k.a.~the ``$\N$-equivariant fundamental group'' of $\U$ \cite{Collingwood:1993}). Endowing $\Ucov$ with the pull-back $\pU^*\wU$ of the Kirillov-Kostant-Souriau 2-form of $\U$, we obtain a primary $\N$-space with \moment map $\pU$.

Now suppose that $\Z$ is any primary $\EFG$-space, or in other words, any symplectic manifold with an $\wZ$-preserving action of $\EFG$. (This is hamiltonian with \moment map $\Z\to\{0\}$.) We may then form the associated ``flat bundle''
   \begin{equation}
      \label{prodFib}
      \Ucov\times_\EFG\Z\xrightarrow{\ \ \PI\ \ }\U
   \end{equation}
   and observe that its total space (the set of orbits $[\ucov,\z]$ of the product action of $\EFG$ on $\Ucov\times\Z$) is naturally a primary $\N$-space over $\U$: the 2-form is the one obtained from $\pU^*\wU+\wZ$ by passage to the quotient, $\N$ acts by $\n([\ucov,\z])=[\n(\ucov),\z]$, and the \moment map is $\PI([\ucov,\z])=\pU(\ucov)$.
   \end{exem}

Conversely:

\begin{theo}[Generalized Barycentric Decomposition]
   \label{decTheo} %\ \newline
   \textup{(i)} Every primary $\N$-space $(\X,\w,\PI)$ over $\U$ is obtained in this way. More precisely~one always has
	   \begin{equation}
	     \label{decFibX}
	     \X=\Ucov\times_\EFG\Z
	     \qquad\ \text{where}\ \qquad
	     \Z=\PI\inv(\a).
	   \end{equation}
	   \textup{(ii)} Two primary $\N$-spaces $(\X_i,\w_i,\PI_i)_{i=1,2}$ over $\U$ are isomorphic iff the fibers $\Z_i=\PI\sub{i}\inv(\a)$ are isomorphic as primary $\EFG$-spaces.    
\end{theo}

\begin{figure}
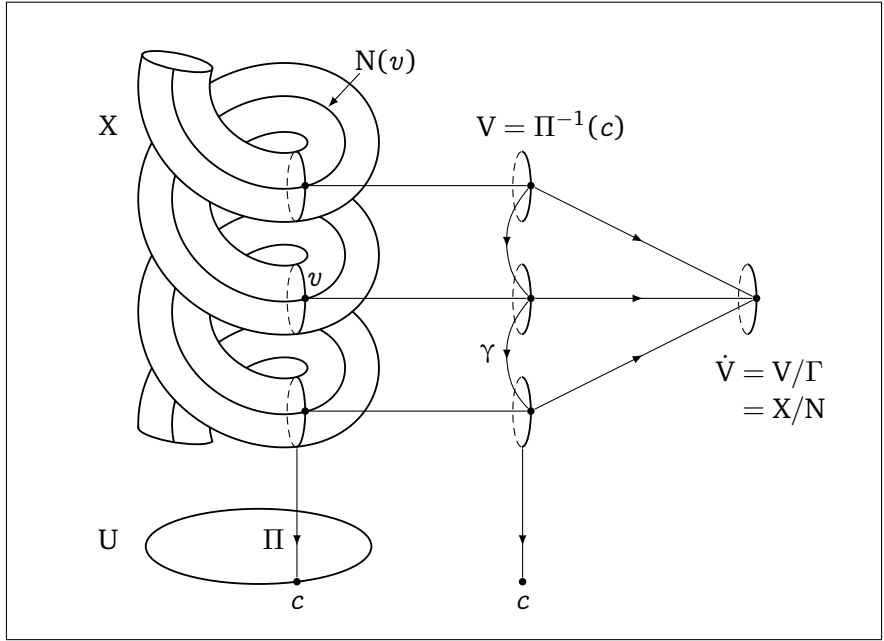
%[h] %  [htbp]figure placement: here, top, bottom, or page
  \centering
  \framebox[1.12\width]{%
  \beginpgfgraphicnamed{fig1}
     \input figure1.tex%
  \endpgfgraphicnamed%
  }
  \caption{Barycentric decomposition of the primary space $\X$.}
  \label{fig:decomposition}
\end{figure}

Implied here is the assertion that $\PI\inv(\a)$ is always naturally a primary $\EFG$\nobreakdash-spa\-ce. This is ensured by the following lemma and corollary, which give the decomposition at the infinitesimal level. 

\begin{lemm}[\textsc{Fig.}\,\ref{fig:decomposition}]
   \label{decLemm}
   Let $(\X,\w,\PI)$ be a primary $\N$-space over $\U=\N(\a)$. Then the fiber $\Z=\PI\inv(c)$ is a symplectic submanifold of $\X$ and we have\textup, for each $\z\in\Z$\textup, the direct sum decomposition into symplectically orthogonal subspaces
   \begin{equation}
      \label{decTX}
      T_\z \X=T_\z\N(\z)\oplus T_\z\Z.
   \end{equation}
\end{lemm}

\begin{proof}
   We know that $\PI:\X\to\LN^*$ remains smooth when regarded as a map $\X\to\U$ where $\U=\N/\K$ has the quotient manifold structure \cite[2.3.12(i)]{Ortega:2004a}. In particular it is (by equivariance) a submersion: hence, its fiber $\Z$ is a submanifold. Moreover, by a well-known property of \moment maps \cites[{}12.152]{Souriau:1970}[Lemma (ii)]{Marsden:1974}, the tangent space to $\Z$ at $\z$ is 
   \begin{equation}
      \label{lemmMW}
      T_\z\Z=\ker(D\PI(\z))=\LN(\z)^\w,
   \end{equation}
   the orthogonal of the tangent space $T_\z\N(\z)=\LN(\z)$ relative to the 2-form $\w$. But $\LN(\z)$ and $\LN(\z)^\w$ span $T_\z\X$: indeed, given $\d\z\in T_\z\X$ we can find a $\ln\in\LN$ such that $D\PI(\z)(\d\z)=\ln(\a)$, and then $\d\z$ is the sum of
   \begin{equation}
      \ln(\z)\in\LN(\z)
      \qquad
      \text{and}
      \qquad
      \d\z - \ln(\z)\in\ker(D\PI(\z)),
   \end{equation}
   where the last inclusion is by equivariance of $\PI$. Now, taking orthogonals on both sides of the relation $T_\z\X=\LN(\z)+\LN(\z)^\w$, we get $\{0\}=\LN(\z)\cap\LN(\z)^\w$. So the sum is direct, and each summand is a symplectic subspace of $T_\z\X$.
\end{proof}

\begin{coro}
   \label{Ko(v)=0}
   In the setting of \eqref{decLemm}\textup, the action of the stabilizer $\K=\N_\a$ on $\Z$ induces the trivial action of\/ $\Ko$\textup, so that $\Z$ carries a natural action of\/ $\EFG=\K/\Ko$.
\end{coro}

\begin{proof}
   The relation $D\PI(\z)(\LK(\z))=\LK(\a)=0$ shows that the tangent space $T_\z\K(\z)=\LK(\z)$ lies in the intersection $\LN(\z) \cap \ker(D\PI(\z))$, which we have just seen is zero. Thus $\LK$ acts trivially on $\Z$, and hence so does the identity component $\Ko$.
\end{proof}

\begin{proof}[Proof of the theorem]
   (i) We claim that the flat bundle \eqref{prodFib} built from $\Z=\PI\inv(\a)$ is isomorphic to $(\X,\w,\PI)$. To see this, let us denote it by $(\Y,\w_\Y,\PI_\Y)$. We have $\N$\nobreakdash-equivariant submersions $\pX,\pY$:
   \begin{equation}
      \label{isomXY}
      \raisebox{-2.5ex}{
      \begin{tikzpicture}[semithick,>=to]
         \node[anchor=east] (11) at (-1.8,0.55) {$\X$};
         \node              (12) at (0,0.55) {$\N\times\Z$};
         \node[anchor=west] (13) at (1.8,0.55) {$\Y$};
         \node[anchor=east] (21) at (-1.8,0) {$\n(\z)=:\x\,$};
         \node              (22) at (0,0) {$\ (\n,\z)\ $};
         \node[anchor=west] (23) at (1.8,0) {$\,\y:=\n([e\Ko,\z])$};
         \draw[<-] (11) to node[auto] {\smash{\scriptsize $\pX$}} (12);
         \draw[->] (12) to node[auto] {\smash{\scriptsize $\pY$}} (13);
         \draw[<-|] (21) -- (22);
         \draw[|->] (22) -- (23);
      \end{tikzpicture}
      }
   \end{equation}
   where $\N$ acts on $\N\times\Z$ via its left action on $\N$. Now since
   % \begin{equation*}
      \begin{align}
         \label{fibers}
         \pY(\n',\z')=\pY(\n,\z)
         % &\Leftrightarrow \n'([\acov,\z'])=\n([\acov,\z])
         % &
         % &
         % \\
         &\Leftrightarrow (\n'\Ko,\z')=(\n\Ko\efg\inv,\efg(\z))
         &
         &\text{for some $\efg\in\EFG$}
         \notag
         \\
         &\Leftrightarrow (\n',\z')=(\n\k\inv,\k(\z))
         &
         &\text{for some $\k\in\K$}
         \notag
         \\
         &\Leftrightarrow \n'(\z')=\n(\z)
         \text{\ \ and\ \ }
         \n\inv\n'\in\K%\rlap{$\K$}
         &
         &
         \\
         &\Leftrightarrow \pX(\n',\z')=\pX(\n,\z),
         &
         &
         \notag
      \end{align}
   % \end{equation*}
   we see that the fibers of $\pX$ and $\pY$ define the same partition of $\N\times\Z$. So \eqref{isomXY} induces an $\N$-equivariant bijection $\X \to \Y$, which is a diffeomorphism by \cite[{}5.9.6]{Bourbaki:1967}, and which pulls $\PI_\Y$ back to $\PI$ since $\PI(\pX(\n,\z))=\n(\a)=\PI_\Y(\pY(\n,\z))$. We must still show that this diffeomorphism is symplectic. To this end, write
   \smash{\raisebox{-3.5pt}{%
   \hspace{-1ex}
   \begin{tikzpicture}[semithick,>=to]
     \node (21) at (-1.7,0) {\raisebox{-5pt}{\smash{$\d\x$}}};
     \node (22) at (0,0)    {\raisebox{-5pt}{\smash{$(\d\n,\d\z)$}}};
     \node (23) at (1.7,0)  {\raisebox{-5pt}{\smash{$\d\y$}}};
     \draw[<-|] (21) -- (22);
     \draw[|->] (22) -- (23);
   \end{tikzpicture}\hspace{-.5ex}}}
   for the diagram tangent to \eqref{isomXY}, i.\,e.
   \begin{equation}
     \d\x=\n_*(\ln(\z)+\d\z)
     \qquad
     \text{and}
     \qquad
     \d\y=\n_*([\ln(e\Ko),\d\z])
   \end{equation}
   where $\ln=\n\inv\d\n$ and the bracket denotes the derived projection $T\Ucov\times T\Z\to T\hspace{1pt}\Y$. Then we have 
   % \begin{equation}
      \begin{align}
         \qquad\qquad
         \w(\d\x,\d'\x)
         &=\w(\ln(\z)+\d\z,\ln'(\z)+\d'\z)
         &
         &\quad\text{since $\n^*\w=\w$}
         \notag
         \\
         &=\w(\ln(\z),\ln'(\z))+\w(\d\z,\d'\z)
         &
         &\quad\text{by \eqref{decTX}}
         \notag
         \\
         &=\<\a,[\ln',\ln]\>+\wZ(\d\z,\d'\z)
         &
         &\quad\text{by \cite[{}11.17$\sharp$]{Souriau:1970}}
         \\
         &=\w_\Y(\d \y,\d'\y)
         &
         &\quad\text{by definition,}
         % &\quad\text{since $\n^*\w_\Y=\w_\Y$,}
         \notag
      \end{align}
   % \end{equation}
   as was to be shown. This completes the proof of (i).
   
   (ii) Isomorphic $\X_i$ give rise to isomorphic $\Z_i$. Indeed, let an isomorphism $F:\X_1\to\X_2$ be given. Then the relation $\PI_1=\PI_2\circ F$ implies that $F$ maps $\Z_1$ onto $\Z_2$ and induces a diffeomorphism $f$ between them. Since the $\Z_i$ have by construction the 2-form and action of $\efg\in\EFG$ induced by the $\w_i$ and by the action of some $\k\in\efg$ on the $\X_i$, it is clear that $f$ is another isomorphism. 
   
   Conversely, isomorphic $\Z_i$ give rise to isomorphic $\X_i$. Indeed, let an isomorphism $f:\Z_1\to\Z_2$ be given. Then in the diagram
   \begin{equation}
      \label{id_x_f}
      \raisebox{-6.4ex}{
   	\begin{tikzpicture}[semithick,>=to]
         \node (11) at (0,1.8) {$\Ucov\times\Z_1$};
         \node (12) at (4.6,1.8) {$\Ucov\times\Z_2$};
         \node (21) at (0,0) {$\Ucov\times_\EFG\Z_1$};
         \node (22) at (4.6,0) {$\Ucov\times_\EFG\Z_2$};
      	\draw[->] (11) to node[auto] {\smash{\scriptsize $\mathrm{id}\times f$}} (12);
      	\draw[densely dashed,->,color=gray] (21) to node[auto] {\color{gray}\scriptsize $F$} (22);
      	\draw[->] (11) to node[left] {\scriptsize $\mathrm{mod}\,\EFG$} (21);
      	\draw[->] (12) to node[auto] {\scriptsize $\mathrm{mod}\,\EFG$} (22);
   	\end{tikzpicture}}
   \end{equation}
   one verifies without trouble that the top arrow is an isomorphism of primary $\N\times\EFG$-spaces, relative to the 2-forms $\pU^*\wU+\w_{\Z_i}$ and actions $(\n,\efg)(\ucov,\z_i)=(\n\ucov\efg\inv,\efg(\z_i))$. So it descends to an isomorphism $F$ between the quotients, of which we already know by \eqref{decFibX} that they are isomorphic to $\X_1$ and $\X_2$.
\end{proof}

\begin{coro1}
   \label{trivCoro}
   If the pair $(\N,\U)$ is in any of the following cases\textup, then every primary $\N$-space over $\U$ splits trivially as $\X=\U\times\Z$\textup:
   \begin{compactenum}[\upshape (i)]
      \item points of\/ $\U$ have connected stabilizers in $\N$\textup;
      \item $\N$ is connected and $\U$ is simply connected\textup;
      \item $\N$ is connected and compact or abelian or nilpotent or exponential\textup;
      \item $\N$ is connected\textup, complex semisimple and $\U$ is a semisimple orbit\textup.
   \end{compactenum}
\end{coro1}

\begin{proof}
   (i) If $\K$ is connected, then $\EFG=\{1\}$ and \eqref{decFibX} reduces to \mbox{$\X=\U\times\Z$}; (ii) implies (i) by the homotopy exact sequence $\pi_1(\U)\to\pi_0(\K)\to\pi_0(\N)$; (iii,~iv) each imply (i) or (ii): see \cites[p.\,4]{Bernat:1972}[{}8.20]{Besse:1987}[{}2.3.4]{Collingwood:1993}.
\end{proof}

\begin{rema}
	The component groups $\EFG$ are also well known when $\N$ is connected solvable or simple: in the former case they are lattices $\ZZ^d$ \cite{Auslander:1971a}; in the latter, finite products of groups $\ZZ_d$, $S_3$, $S_4$, $S_5$ \cite{King:1992,Collingwood:1993}. Lest the reader conclude that they are always dull, we note here that every finite group occurs for some coadjoint orbit of some connected Lie group \cite[Prop.~1(iii)]{Pukanszky:1978}.
\end{rema}

As soon as $\EFG\ne\{1\}$ it becomes an issue to detect which primary spaces do or don't split. For this we have the following simple criteria.

\begin{coro2}
   \label{nontrivCoro}
   Let $\X$ be primary with fiber $\Z$. 
   % \begin{compactenum}[\upshape (i)]
   \textup{(i)}
   % \item
   If $\X$ splits trivially\textup, then $\EFG$ acts trivially on $\Z$. 
   \textup{(ii)}
   % \item
   If $\X$ splits nontrivially\textup, then $\Z$ is not connected.
   % \end{compactenum}
\end{coro2}

\begin{proof}
   (i) The hypothesis means that $(X,\w,\PI)\cong(\U\times \T,\w_\U+\w_\T,\textrm{pr}_1)$ for some $\T$ on which $\N$ acts trivially. By (\ref{decTheo}ii) it follows that $\Z\cong\{\a\}\times\T$, on which $\EFG$ acts trivially. (ii) Likewise if $(X,\w,\PI)\cong(\bar\U\times \T,\pi^*\w_\U+\w_\T,\pi\circ\textrm{pr}_1)$ for some nontrivial covering $\pi$ then $\Z\cong\pi\inv(c)\times\T$, which is not connected.
\end{proof}

\begin{remas}
	(i) We emphasize that --- thanks essentially to the rigidity of the \moment map --- these criteria are very much simpler than the topological arguments one would need to decide whether $\U\times_\EFG\Z$ is trivial \emph{as a bundle}, or whether it is \emph{homeomorphic} to a product $\bar\U\times\T$.
	
	(ii) It may also be worth noting that Theorem \eqref{decTheo} places any primary space \emph{half-way} between two products:
	\begin{equation}
	   \label{half-way}
	   \raisebox{-5ex}{
		\begin{tikzpicture}[semithick,>=to]
	   	\node (a) at (-2.3,0.65) {$\Ucov\times\Z$};
	   	\node (b) at (0,0) {$\X$};
	   	\node (c) at (2.3,-0.65) {$\U\times\Zdot,$};
	   	\draw[->] (a) to node[auto] {\scriptsize $\mathrm{mod}\ \EFG$} (b);
	   	\draw[->] (b) to node[auto] {\scriptsize $\PI\times\mathrm{mod}\ \N$} (c);
		\end{tikzpicture}}
	\end{equation}
	where $\Zdot$ is the quotient $\X/\N=\Z/\EFG$ (\textsc{Fig.\,\ref{fig:decomposition}}). However, the upper product here does not naturally carry an action of $\N$; and the lower one is only a (trivially split) primary space insofar as $\Zdot$ is a manifold, which fails in general: see e.g. Example \eqref{Kodaira-Thurston} below, where $\Zdot$ is not even an orbifold.
\end{remas}

\begin{exem}[Reduced spaces of Kazhdan \emph{et al.}]
   \label{KKS}
   An important source of primary spaces is the construction of Kazhdan-Kostant-Sternberg \cite[p.\,482]{Kazhdan:1978}: if $(\Y,\wY,\PIY)$ is an arbitrary hamiltonian $\N$-space, then under appropriate transversality hypotheses one can form the \emph{reduced space at $\U$},
   \begin{equation}
      \label{redKKS}
      \X=\PIY\inv(\U)/\ker(\wY)
   \end{equation}
   (the quotient of $\PIY\inv(\U)$, which is coisotropic, by its characteristic foliation). This is naturally a primary $\N$-space over $\U$, to which the above applies. In this setting we note that \eqref{decLemm}, (\ref{trivCoro}i) and a version of (\ref{decTheo}i) are found respectively in \cites[Prop.~1.1]{Kazhdan:1978}[Thm~26.6]{Guillemin:1984}[Thm~1]{Odzijewicz:1984}. (Our treatment avoids the assumption, made in these references, that $\U\hookrightarrow\LN^*$ be an embedding.) Also, the spaces $\Zdot$ and $\U\times\Zdot$ of \eqref{half-way} are none other than the quotients 
   \begin{equation}
      \PIY\inv(\a)/\N_\a
      \qquad\quad\text{and}\qquad\quad
      \PIY\inv(\U)/R
   \end{equation}
   of Marsden-Weinstein \cite{Marsden:1974} and Marle \cite{Marle:1976} (to which we refer for the definition of the equivalence relation $R$). In particular Marle's quotient always splits trivially, beyond the cases elucidated in \cite[Cor.~p.\,257]{Marle:1976}.
   \end{exem}

%%%%%%%%%%%%%%%%%%%%%%%%%%%%%%%%%%%%%%%%%%%%%%%%%%%%%%%%%%%%%%%%%%%%%%%%
\section{Extensions of a primary $\N$-space}
%%%%%%%%%%%%%%%%%%%%%%%%%%%%%%%%%%%%%%%%%%%%%%%%%%%%%%%%%%%%%%%%%%%%%%%%
\label{section2}

From here on we assume again that $\N$ is a closed normal subgroup of a Lie group $\G$. Our aim is to describe those $\N$-primary spaces $(\X,\w,\PI)$ over $\U$ that arise by \emph{restriction} of an action of the larger group --- i.e.~$\X$ admits a hamiltonian $\G$\nobreakdash-action, with equivariant \moment map $\PHI$, such that the following diagrams commute:
\begin{equation}
   \label{restriction}
   \raisebox{-3.5ex}{
	\begin{tikzpicture}[semithick,>=to]
   	\node (Diff) at (1.5,0.5) {D\rlap{iff$(\X,\w)$}};
   	\node (G) at (0,1) {$\G$};
   	\node (N) at (0,0) {$\N$};
   	\draw[->,dashed] (G) to node[above] {\scriptsize\qquad$\G$-action} (Diff);
   	\draw[->] (N) to node[below] {\scriptsize\qquad$\N$-action} (Diff);
   	\draw[->] (N) to (G);
   	\node (X) at (5,0.5) {$\X\ $};
   	\node (g*) at (6.5,1) {$\LG\rlap{$^*$}$};
   	\node (n*) at (6.5,0) {$\LN\rlap{$^*$.}$};
   	\draw[->,dashed] (X) to node[above] {\scriptsize $\PHI$} (g*);
   	\draw[->] (X) to node[below] {\scriptsize $\PI$} (n*);
   	\draw[->] (g*) to (n*);
	\end{tikzpicture}}
\end{equation}
Suppose for a moment that such an extended action exists. Because it normalizes $\N$, $\G$ has coadjoint actions on both $\LG^*$ and $\LN^*$, relative to which the second diagram \eqref{restriction} is equivariant. In particular $\G$ preserves the image of $\PI$, so $\U$ must belong to the set $(\LN^*/\N)^\G$ of \emph{$\G$-invariant} coadjoint orbits of $\N$. Under this assumption we have the following result, which in effect reduces the extension problem to the case where $\U$ is a point.

\begin{theo}
   \label{extTheo}
   Let $\N$ be normal in $\G$\textup, pick an orbit $\U=\N(\a)\in(\LN^*/\N)^\G$\textup, and write $\K$ and $\L$ for the stabilizers $\N_\a$ and $\G_\a$. Then the correspondence $\X\rightleftarrows\Z$ of Theorem \eqref{decTheo} naturally induces another between
\begin{equation*}
   \begin{gathered}
      \textit{$\N$-primary hamiltonian $\G$-spaces}\\[-.55ex]
      \textit{$(\X,\w,\PHI)$ over $\U$}
   \end{gathered}
   \quad\rightleftarrows\quad
   \begin{gathered}
      \textit{$\K$-primary hamiltonian $\L$-spaces}\\[-.55ex]
      \textit{$(\Z,\wZ,\PSI)$ over the point $\{\a_{|\LK}\}$.}
   \end{gathered}
\end{equation*}
\end{theo}

The proof will make repeated use of the following elementary observation, which essentially describes the trivial case where $\U=\{0\}$.

\begin{lemm} 
   \label{N_fixes_ann_n}
   If\/ $\N$ is normal in $\G$\textup, the annihilator 
   \mbox{$\ann_{\LG^*}(\LN)=\ker(\LG^*\to\LN^*)$} is 
   \begin{enumerate}[\upshape (i)]
   	\item invariant under the coadjoint action of\/ $\G$ on $\LG^*$\textup,
	   \item pointwise fixed by the action of\/ $\N$\textup, 
	   \item as a $\G/\N$-module\textup, canonically isomorphic to $(\LG/\LN)^*$.
   \end{enumerate}
\end{lemm}

\begin{proof}
   Write $\INT\h(\g)=\h\g\h\inv$ for the conjugation action of $\h\in\G$ on elements (or subsets) of $\G$. Since $\N$ is normal, $\INT\h$ maps each coset $\g\N$ to another such coset,
   \begin{equation}
      \label{conjugation}
      \INT\h(\g\N)=\h\g\N\h\inv=\INT\h(\g)\N.
   \end{equation}
   So we have an induced conjugation action on $\G/\N$ which is just the lift of its conjugation action on itself, for \eqref{conjugation} can also be written $(\h\N)\g\N(\h\N)\inv$. In other words, the exact sequence
   \begin{equation}
      \raisebox{-4.4ex}{
      \begin{tikzpicture}[semithick,>=to,->]
         \def\l{0.85}
         \def\cola{-3.8*\l}
         \def\colb{-2.0*\l}
         \def\colc{0.0*\l}
         \def\cold{2.0*\l}
         \def\cole{3.7*\l}
         \node (1) at (\cola,0) {$1$};
         \node (2) at (\colb,0) {$\N$};
         \node (3) at (\colc,0) {$\G$};
         \node (4) at (\cold,0) {\raisebox{-2ex}{$\G/\N$}};
         \node (5) at (\cole,0) {$1$};
         \node (4b) at (\cold,0) {$\phantom{\G}$};
         \path (1) edge (2) (2) edge (3) (3) edge (4) (4) edge (5)
               (2) edge [->,out=245,in=295,below,looseness=3] 
                    node {$\scriptstyle\G$} (2)
               (3) edge [->,out=245,in=295,below,looseness=3] 
                    node {$\scriptstyle\G$} (3)
               (4b) edge [->,out=245,in=295,below,looseness=3] 
                    node {$\scriptstyle\G$} (4b);
      \end{tikzpicture}}
   \end{equation}
   is $\G$-equivariant, and the third $\G$-action is in fact an action of $\G/\N$ (so $\N$ acts trivially). Deriving at $e$ and dualizing, we get exactness and equivariance of  
   \begin{equation}
      \label{dualSequence}
      \raisebox{-4.4ex}{
      \begin{tikzpicture}[semithick,>=to,->]
         \def\l{0.85}
         \def\cola{-3.8*\l}
         \def\colb{-2.0*\l}
         \def\colc{0.0*\l}
         \def\cold{2.0*\l}
         \def\cole{3.7*\l}
         \node (1) at (\cola,0) {$0$};
         \node (2) at (\colb,0) {$(\LG/\LN)^*$};
         \node (3) at (\colc,0) {$\LG^*$};
         \node (4) at (\cold,0) {\raisebox{.4ex}{$\LN^*$}};
         \node (5) at (\cole,0) {$0$};
         \node (2b) at (\colb-0.1*\l,0) {$\phantom{\G}$};
         \node (3b) at (\colc-0.1*\l,0) {$\phantom{\G}$};
         \node (4b) at (\cold-0.1*\l,0) {$\phantom{\G}$};
         \path (1) edge (2) (2) edge (3) (3) edge (4) (4) edge (5)
               (2b) edge [->,out=245,in=295,below,looseness=3] 
                    node {$\scriptstyle\G$} (2b)
               (3b) edge [->,out=245,in=295,below,looseness=3] 
                    node {$\scriptstyle\G$} (3b)
               (4b) edge [->,out=245,in=295,below,looseness=3] 
                    node {$\scriptstyle\G$} (4b);
      \end{tikzpicture}}
   \end{equation}
   relative to the 3 coadjoint actions, of which the first is in fact lifted from $\G/\N$ (so $\N$ acts trivially). This is precisely what the lemma claims.
   \end{proof}

\begin{proof}[Proof of the theorem]
   Suppose that $(\X,\w,\PHI)$ is given. The $\G$-equivari\-ance of $\PI$ ensures that the stabilizer $\L$ of $\a$ preserves $\Z=\PI\inv(\a)$. So $\Z$ becomes a hamiltonian $\L$-space, whose \moment map $\PSI$ we can define by the commutativity of   
   \begin{equation}
      \label{defPsi}
      \raisebox{-5.4ex}{
   	\begin{tikzpicture}[semithick,>=to]
         \node (10) at (-1,1.5) {$\X$};
         \node (11) at (1,1.5) {$\LG^*$};
         \node (12) at (3,1.5) {$\LN^*$\rlap{\ $\ni\a$}};
         \node (20) at (-1,0) {$\Z\vphantom{\int}$};
         \node (21) at (1,0) {$\LL^*$};
         \node (22) at (3,0) {$\LK^*$\rlap{\ $\ni\a_{|\LK}$.}};
         \draw[->] (11) -- (12);
         \draw[->] (21) -- (22);
         \draw[->] (11) -- (21);
         \draw[->] (12) -- (22);
         \draw[->] (10) to node[above] {\raisebox{0pt}{\scriptsize $\PHI$}} (11);
         \draw[->] (20) to node[above] {\raisebox{0pt}{\scriptsize $\PSI$}} (21);
         \draw[->] (20) -- (10);
         \draw (-1.1,0.353) arc (180:360:0.05);
   	\end{tikzpicture}}
   \end{equation}
   Because $\Z$ sits above $\a\in\LN^*$ by construction, it follows that it sits above $\a_{|\LK}$ under the composition $\Z\to\LL^*\to\LK^*$, which is the \moment map for the induced $\K$-action; in other words $\Z$ is $\K$-primary over $\{\a_{|\LK}\}$, as claimed.

   Conversely suppose that $(\Z,\wZ,\PSI)$ is given. We must construct on $\X:=\Ucov\times_\EFG\Z$ a $\G$-action and \moment map $\PHI$ satisfying \eqref{restriction}. The rough (but wrong) idea is that $\G$ acts separately on $\U$ and $\Z$, with \moment maps $\phiU$ and $\phiZ$ that add up to an equivariant $\PHI$. To make this correct we must in fact take a detour via $\Ucov$, where $\phiU$ will be defined, and an attendant $\Gcov$ which actually acts on $\Ucov$ and $\Z$. We do this in the following steps, where $\N\rtimes\L$ denotes the semidirect product with law $(\n,\l)(\n',\l')=(\n\AD\l(\n'),\l\l')$ (notation \eqref{conjugation}).

   1. 
   We can define a covering extension $\Gcov$ of $\G$ by $\EFG$ by the following commutative diagram of exact sequences, where $\antidiag\k=(\k\inv,\k)$ and $\proj(\n,\l)=\n\l$:
   \begin{equation}
      \label{Gcov}
      \raisebox{-14.5ex}{
      \begin{tikzpicture}[semithick,>=to,->]
         \def\l{0.85}
         \def\rowa{3.0*\l}
         \def\rowb{1.6*\l}
         \def\rowc{0.0*\l}
         \def\rowd{-1.6*\l}
         \def\rowe{-3.0*\l}
         \def\cola{-3.8*\l}
         \def\colb{-2.0*\l}
         \def\colc{0.0*\l}
         \def\cold{2.0*\l}
         \def\cole{3.7*\l}
         \node (12) at (\colb,\rowa) {$1$};
         \node (13) at (\colc,\rowa) {$1$};
         \node (14) at (\cold,\rowa) {};
         \node (21) at (\cola,\rowb) {$1$};
         \node (22) at (\colb,\rowb) {$\antidiag\Ko$};  % {$\Ko$};
         \node (23) at (\colc,\rowb) {$\antidiag\Ko$};  % {$\Ko$};
         \node (24) at (\cold,\rowb) {$1$};
         \node (25) at (\cole,\rowb) {};
         \node (31) at (\cola,\rowc) {$1$};
         \node (32) at (\colb,\rowc) {$\antidiag\K$};   % {$\K$};
         \node (33) at (\colc,\rowc) {$\phantom{\N\times\G}$};
         \node (33bis) at (0.08*\l,\rowc) {$\N\rtimes\L$};
         \node (34) at (\cold,\rowc) {$\G$};
         \node (35) at (\cole,\rowc) {$1$};
         \node (41) at (\cola,\rowd) {$1$};
         \node (42) at (\colb,\rowd) {$\EFG$};
         \node (43) at (\colc,\rowd) {\raisebox{2.5pt}{$\cov\G$}};
         \node (44) at (\cold,\rowd) {$\G$};
         \node (45) at (\cole,\rowd) {$1$};
         \node (52) at (\colb,\rowe) {$1$};
         \node (53) at (\colc,\rowe) {$1$};
         \node (54) at (\cold,\rowe) {$1$\smash{\rlap{.}}};
         \foreach \from/\to in
            {     12/22,      13/23,
            21/22,      22/23,      23/24,
                  22/32,      23/33,      24/34,
            31/32,      32/33,                  34/35,
                  32/42,                  34/44,
                  42/52,      43/53,      44/54}
            \draw (\from) -- (\to);
         \foreach \from/\to in 
            {41/42,     42/43,      43/44,      44/45}
            \draw[densely dashed,color=gray] (\from) -- (\to);
         \draw (33) to node[right]{\color{gray}$\scriptstyle\bracl\cdot,\cdot\bracr$} (43);
         \draw (33bis) to node[pos=0.4,above]{$\scriptstyle\proj$} (34);
      \end{tikzpicture}}
   \end{equation}
   Indeed, $\proj$ clearly has kernel $\antidiag\K$ and is onto: given $\g\in\G$, we can find~$\n\in\N$ such that $\g(\a)=\n(\a)$, whence $g=\proj(\n,\n\inv\g)$. So the top and middle row are exact. In particular $\N\rtimes\L$ normalizes $\antidiag\K$, hence also its identity~component $\antidiag\Ko$ \cite[{}7.16]{Hewitt:1963}. Writing $\Gcov$ for the resulting quotient and
   \begin{equation}
      (n,l)\mapsto\bracl n,l\bracr
      \qquad\text{and}\qquad
      \gcov\mapsto\g
   \end{equation}
   for the projections $\N\rtimes\L\to\Gcov$ and $\Gcov\to\G$, we have exact columns too, so the nine lemma (or the second isomorphism theorem \cite[{}5.35]{Hewitt:1963}) completes the diagram. Before continuing we observe that $\bracl\cdot,e\bracr$ and $\bracl e,\cdot\bracr$ map $\N$ and $\L$ injectively onto isomorphic copies which we shall denote
   \begin{equation}
      \label{NcovLcov}
   	\Ncov = \bracl\N,e\bracr,
   	\qquad\qquad
   	\Lcov = \bracl e,\L\bracr.
   \end{equation}
   (The gain inside $\Gcov$ is that unlike $\N\cap\L$ ($=\K$), the intersection $\Ncov\cap\Lcov$ ($=\bracl\Ko,e\bracr$) is always connected.)
  
   2.
   We claim that the coadjoint action of $\G$ on $\U$ lifts to an action of $\Gcov$ on $\Ucov$. To see this, we note that \eqref{Gcov} contains the relations $\G=\N\L$ and $\Gcov=\Ncov\Lcov$. Therefore the first isomorphism theorem \cite[{}5.33]{Hewitt:1963} implies that all horizontal maps in the following commutative diagrams are bijections:
\begin{equation}
   \label{isom1}
   \phantom{U=\ }
   \raisebox{-6.4ex}{
	\begin{tikzpicture}[semithick,>=to]
      \node (11) at (0,1.8) {\llap{$\Ucov=$} $\N/\Ko$};
      \node (12) at (3.5,1.8) {$\Gcov/\Lcov$};
      \node (21) at (0.9,1.3) {\scriptsize$\n\Ko$};
      \node (22) at (2.5,1.3) {\scriptsize$\bracl\n,\L\bracr$};
      \node (31) at (0.9,0.5) {\scriptsize$\n\K$};
      \node (32) at (2.5,0.5) {\scriptsize$\n\L$};
      \node (41) at (0,0) {\llap{$\U=$} $\N/\K$};
      \node (42) at (3.5,0) {$\G/\L$};
   	\draw[->] (11) to node[auto] 
   	{\smash{\raisebox{-2pt}{\color{gray}$\scriptscriptstyle\sim$}}} (12);
   	\draw[|->,thin] (21) -- (22);
   	\draw[|->,thin] (31) -- (32);
   	\draw[->] (41) to node[auto] 
   	{\smash{\raisebox{-2pt}{\color{gray}$\scriptscriptstyle\sim$}}} (42);
   	\draw[->] (11) to node[left] {\scriptsize$\pU$} (41);
   	\draw[|->,thin] (21) -- (31);
   	\draw[->] (12) -- (42);
   	\draw[|->,thin] (22) -- (32);
      \node (13) at (5.5,1.8) {$\L/\Ko$};
      \node (14) at (9,1.8) {$\Gcov/\Ncov$};
      \node (23) at (6.5,1.3) {\scriptsize$\l\Ko$};
      \node (24) at (8,1.3) {\scriptsize$\bracl\N,\l\bracr$};
      \node (33) at (6.5,0.5) {\scriptsize$\l\K$};
      \node (34) at (8,0.5) {\scriptsize$\l\N$};
      \node (43) at (5.5,0) {$\L/\K$};
      \node (44) at (9,0) {$\G/\N$.};
      \draw[->] (13) to node[auto] 
      {\smash{\raisebox{-2pt}{$\color{gray}\scriptscriptstyle\sim$}}} (14);
      \draw[|->,thin] (23) -- (24);
      \draw[|->,thin] (33) -- (34);
      \draw[->] (43) to node[auto] 
      {\smash{\raisebox{-2pt}{$\color{gray}\scriptscriptstyle\sim$}}} (44);
      \draw[->] (13) -- (43);
      \draw[|->,thin] (23) -- (33);
      \draw[->] (14) -- (44);
      \draw[|->,thin] (24) -- (34);
	\end{tikzpicture}}
\end{equation}
Now the first diagram identifies our $\N$-equivariant covering  \eqref{Ucov}, on the left, with the $\Gcov$\nobreakdash-equivariant covering on the right. This equips $\Ucov$ with an action of $\Gcov$ which fits the bill and works out, since $\bracl\n,\l\bracr\bracl\n',\L\bracr =\bracl\n\AD\l(\n'),\L\bracr$, as
\begin{equation}
   \label{actionGcovUcov}
   \bracl\n,\l\bracr(\ucov)=n\AD\l(\ucov),\rlap{\qquad$\ucov=n'\Ko\in N/\Ko$.}
\end{equation}

3. 
This action \eqref{actionGcovUcov} preserves $\wUcov=\pU^*\wU$ and admits a (not necessarily equivariant) \moment map $\phiU:\Ucov\to\LG^*$ given by 
\begin{equation}
   \label{momentGcovUcov}
   \phiU(\ucov)=\ucov(\acheck),
\end{equation}
where we fix henceforth an element $\acheck\in\LG^*$ projecting to $\a\in\LN^*$. Indeed, first we note that \eqref{momentGcovUcov} is well-defined, i.e.~all elements of a coset $\ucov=n\Ko$ have the same effect on $\acheck$. This is because any two differ by an element of $\Ko$, which acts trivially by \eqref{Ko(v)=0} applied to the orbit $\G(\acheck)$ (which is clearly $\N$\nobreakdash-primary over $\U$). Next, \eqref{momentGcovUcov} is a \moment map because we have, for $\ucov=n\Ko$ and all $(\ln,\lg)\in\LN\times\LG$,
% \begin{equation*}
   \begin{align}
   \smash[b]{\left[\I_\lg\wUcov
   +d\<\phiU(\cdot),\lg\>\right]}(\ln(\ucov))
   &=\left.\wUcov(\lg(\ucov),\ln(\ucov))
   +\tfrac{d}{dt}\<\phiU(\e{t\ln}(\ucov)),\lg\>\right|_{t=0}
   \notag
   \\
   &=\left.\wU(\lg(\u),\ln(\u))
   +\tfrac{d}{dt}\<\e{t\ln}(\n(\acheck)),\lg\>\right|_{t=0}
   \notag
   \\
   &=\<\n(\a),[\ln,\lg]\>+\<\n(\acheck),[\lg,\ln]\>
   \\
   &=0.
   \notag
   \end{align}
% \end{equation*}
If $\Gcov$ is disconnected we still need to check that the action \eqref{actionGcovUcov} preserves $\wUcov$. Since $\pU$ in \eqref{isom1} intertwines it with the coadjoint action of $\G$ on $\U$, it suffices to show that $g^*\wU=\wU$, which is true because we have, for all $\ln,\ln'\in\LN$,
% \begin{equation*}
   \begin{align}
      \g^*\wU(\ln(\u),\ln'(\u))
      &=\wU(\g_*(\ln(\u)),\g_*(\ln'(\u)))
      \notag
      \\
      &=\wU((\Ad_\g(\ln))(\g(\u)),(\Ad_\g(\ln'))(\g(\u)))
      \notag
      \\
      &=\<\g(\u),[\Ad_\g(\ln'),\Ad_\g(\ln)]\>
      \\
      % &=\wU((\Ad_\g\ln)(\g(\u)),(\Ad_\g\ln')(\g(\u)))\\
      % &=\<\g(\u),[\Ad_\g\ln',\Ad_\g\ln]\>\\
      &=\<\u,[\ln',\ln]\>
      \notag
      \\
      &=\wU(\ln(\u),\ln'(\u)).
      \notag
   \end{align}
% \end{equation*}

   4. 
   The given action of $\L$ on $\Z$ induces an $\wZ$-preserving action of $\Gcov$ on $\Z$, viz.
   \begin{equation}
      \label{actionGcovZ}
      \bracl\n,\l\bracr(\z)=\l(\z).
   \end{equation}
   Indeed, this is well-defined because $\bracl\n,\l\bracr$ determines $\l\in\L$ up to multiplication by an element of $\Ko$, which we know acts trivially on $\Z$ (by Cor.~\ref{Ko(v)=0}, applied this time to $\Z$ as a $\K$-primary $\L$-space). Moreover it preserves $\wZ$, since the given action of $\L$ does.
   
   5. 
   This action \eqref{actionGcovZ} admits a (not necessarily equivariant) \moment map $\phiZ:\Z\to\LG^*$ given by the following formula, where $\pGLK^*$ denotes the injection \smash{$(\LL/\LK)^*\xrightarrow[\smash{\raisebox{3.5pt}{$\scriptscriptstyle{\eqref{isom1}}$}}]{}(\LG/\LN)^*\xrightarrow[\smash{\raisebox{3.5pt}{$\scriptscriptstyle{\eqref{dualSequence}}$}}]{}\LG^*$} viewed as a map $\ann_{\LL^*}(\LK)\to\ann_{\LG^*}(\LN)$ \eqref{N_fixes_ann_n}:
   \begin{equation}
      \label{momentGcovZ}
      \phiZ(\z)=\pGLK^*(\PSI(\z)-\acheck_{|\LL}).
   \end{equation}
   This is well-defined because our hypothesis that $\Z$ is $\K$-primary over $\{\a_{|\LK}\}$ ensures $\PSI(\z)-\acheck_{|\LL}\in\ann_{\LL^*}(\LK)$. To see that it is a \moment map, we note that (\ref{Gcov}) implies that every $\lg\in\LG$ can be written $\ln+\ll$ for some $(\ln,\ll)\in\LN\rtimes\LL$. Then,
   % \begin{equation*}
	   \begin{align}
	      \I_{\lg}\wZ+d\<\phiZ(\cdot),\lg\>
         &=\I_{\ll}\wZ+d\<\phiZ(\cdot),\lg\>
         &
         &\quad\text{by \eqref{actionGcovZ}}
         \notag
         \\
         &=\I_{\ll}\wZ+d\<\phiZ(\cdot),\ll\>
         &
         &\quad\text{since $\text{range}(\phiZ)\subset\ann_{\LG^*}(\LN)$}
         \notag
         \\
         &=\I_{\ll}\wZ+d\<\PSI(\cdot),\ll\>
         &
         % &\quad\text{by \eqref{momentGcovZ}}\\
         &\quad\text{since $\<\acheck,\ll\>$ is constant}
         \\
         &=0
         &
         &\quad\text{since $\PSI$ is a \moment map.}
         \notag
	   \end{align}
   % \end{equation*}

      6. 
      Putting the actions \eqref{actionGcovUcov} and \eqref{actionGcovZ} together, we get an action of $\Gcov$ on $\Ucov\times\Z$ which we claim is hamiltonian, i.e.~its \moment map $\phiU+\phiZ$ is $\Gcov$-equivariant. In other words we will show that $\theta_\Ucov+\theta_\Z=0$ where
      \begin{equation}
         \label{theta}
	      \begin{aligned}
	         \theta_\Ucov(\gcov) &:= \phiU(\gcov(\ucov))-g(\phiU(\ucov)),\\
	         \theta_\Z(\gcov) &:= \rlap{$\,\phiZ(\gcov(\z))$}\phantom{\phiU(\gcov(\ucov))}-g(\phiZ(\z)).
	      \end{aligned}
      \end{equation}
      To work out these `non-equivariance cocycles' (which as the notation anticipates, will be independent of $\ucov$ and $\z$; note that for disconnected $\X$ this is emphatically \emph{not} guaranteed by the general theory \cites[{}11.17]{Souriau:1970}[{}4.5.21]{Ortega:2004a}), we pick $\gcov=\bracl\n,\l\bracr$ projecting to $\g=\n\l$ and compute, on the $\Ucov$ side,
      \begin{align}
         \theta_\Ucov(\gcov) 
         &= \phiU(\n\AD\l(\ucov))-(\n\l)(\phiU(\ucov))
         &
         &\text{by \eqref{actionGcovUcov}}
         \notag\\
         &= (\n\AD\l(\ucov)-\n\l\ucov)(\acheck)
         &
         &\text{by \eqref{momentGcovUcov}}
         \notag\\
         &= \n\AD\l(\ucov)(\acheck-\l(\acheck))
         \notag\\
         &= \acheck-\l(\acheck)
         &
         \label{thetaU}
         &\text{by Lemma (\ref{N_fixes_ann_n}ii),}
      \intertext{which lemma applies because $\l\in\G_\a$ implies $\<\acheck-\l(\acheck),\LN\>=\<\a-\l(\a),\LN\>=0$, i.e.~we have
      \begin{equation}
         \label{cocycle}
      	\acheck-\l(\acheck)\in\ann_{\LG^*}(\LN).
      \end{equation}
      Meanwhile on the $\Z$ side we find, for all $\lg=\ln+\ll$ in $\LG=\LN+\LL$,}
         \<\theta_\Z(\gcov),\lg\>
         &= \<\phiZ(\l(\z))-(\n\l)(\phiZ(\z)),\lg\>
         &
         &\text{by \eqref{actionGcovZ}}
         \notag\\
         &= \<\phiZ(\l(\z))-\l(\phiZ(\z)),\lg\>
         &
         &\text{by (\ref{N_fixes_ann_n}i,\,ii)}
         \notag\\
         &= \<\phiZ(\l(\z))-\l(\phiZ(\z)),\ll\>
         &
         &\text{since $\text{range}(\phiZ)\subset\ann_{\LG^*}(\LN)$}
         \notag\\
         &= \<\PSI(\l(\z))-\acheck_{|\LL}-\l(\PSI(\z)-\acheck_{|\LL}),\ll\>
         &
         &\text{by \eqref{momentGcovZ}}
         \\
         &= \<\l(\acheck)-\acheck,\ll\>
         &
         &\text{since $\PSI$ is $\L$-equivariant}
         \notag\\
         &= \<\l(\acheck)-\acheck,\lg\>
         &
         &\text{by \eqref{cocycle}.}
         \notag
      \end{align}
      Thus we have $\theta_\Ucov+\theta_\Z=0$, which proves that $\phiU+\phiZ$ is $\Gcov$-equivariant.       

      7.
      Next we consider the subgroup $\EFG=\{\bracl\k\inv,\k\bracr:\k\in\K\}$
in \eqref{Gcov} and observe that its orbits in $\Ucov\times\Z$, under the action just considered, are exactly the points of $\X=\Ucov\times_\EFG\Z$. Since $\Gcov$ normalizes $\EFG$, it follows that its action takes $\EFG$-orbit to $\EFG$-orbit and hence descends to an action of $\G=\Gcov/\EFG$ on $\X$, viz.
\begin{equation}
   \label{actionGX}
	\n\l([\ucov,\z])=[\n\AD\l(\ucov),\l(\z)].
\end{equation}
Moreover the $\EFG$-equivariance of $\phiU+\phiZ$ means precisely that this map is constant on $\EFG$-orbits, hence descends to an equivariant \moment map $\PHI:\X\to\LG^*$. Tracing through its construction yields explicitly, 
with notation as above,
\begin{equation}
   \label{momentGX}
	\<\PHI([\ucov,\z]),\ln+\ll\>=\<\pU(\ucov),\ln\>+\<\ucov(\acheck)-\acheck,\ll\>+\<\PSI(\z),\ll\>.
\end{equation}
Finally we check that $\PHI$ does not depend on the choice of $\acheck$. Indeed if $\acheck'\in\LG^*$ is another choice projecting to $\a\in\LN^*$, then the difference $z=\acheck'-\acheck$ lies in $\ann_{\LG^*}(\LN)$. So (\ref{N_fixes_ann_n}ii) implies $\ucov(z)-z=0$, which shows that the change leaves \eqref{momentGX} unaltered. This completes the proof.
\end{proof}

As in \eqref{decTheo}, $\X$ and $\Z$ determine each other to within isomorphism. More precisely we have the following proposition, whose proof is straightforward and left to the reader.
    
\begin{prop}
   Under the correspondence of Theorem \eqref{extTheo}\textup,
   \begin{enumerate}[\upshape (i)]
   	\item Two $\N$-primary hamiltonian $\G$-spaces $\X_1,\X_2$ over $\U$ are isomorphic iff the corresponding $\Z_i=\PI\sub{i}\inv(\a)$ are isomorphic as hamiltonian $\L$\nobreakdash-spaces. 
   	\item $\X$ is $\G$-homogeneous iff\/ $\Z$ is $\L$-homogeneous\textup, and $\X$ is the coadjoint orbit $\G(\acheck)$ \textup(where $\smash{\acheck_{|\LN}}=c$\textup) iff\/ $\Z$ is the coadjoint orbit $\L(\smash{\acheck_{|\LL}})$. 
   \end{enumerate}
\end{prop}

\section{Kœnig's theorem and the Mackey obstruction}
%%%%%%%%%%%%%%%%%%%%%%%%%%%%%%%%%%%%%%%%%%%%%%%%%%%%%%%%%%%%%%%%%%%%%%%%

In this section we spell out three corollaries of Theorem \eqref{extTheo} and its proof. We~keep the notations $\K$, $\EFG$, $\Ucov$, $\Gcov$, $\Ncov$, $\Lcov$ established in (\ref{Ucov}--\ref{EFG}) and (\ref{Gcov}--\ref{NcovLcov}). The first corollary merely records how we reconstructed the \moment map $\PHI$ from two independent pieces:

\begin{coro1}[Generalized Kœnig Theorem]
   \label{KoenigTheo}
   Let $(\X,\w,\PHI)$ be an $\N$\nobreakdash-pri\-mary hamiltonian $\G$-space over $\U$ with fiber $\Z$. Then not only does $\G$ act on $\X=\Ucov\times_\EFG\Z$\textup, but the product $\Gcov\times\Gcov$ acts on $\Ucov\times\Z$ with \moment map $(\phiU, \phiZ)$ such that
   \begin{equation}
      \Phi([\ucov,\z]) = \phiU(\ucov) + \phiZ(\z).
   \end{equation}
   The second action is really an action of $\Gcov/\Ncov$ with \moment map $\phiZ:\Z\to\ann_{\LG^*}(\LN)$.
\end{coro1}

Unlike $\PHI$, the maps $\phiU$, $\phiZ$, $\theta_\Ucov$, $\theta_\Z$ do depend on the choice of $\acheck$ such that $\smash{\acheck_{|\LN}}=c$. The next corollary extracts the intrinsic part of this construction:
    
\begin{coro2}
   \label{obstCoro}
   Attached to each $\U=\N(\a)\in(\LN^*/\N)^\G$ is a well-defined 
   % symplectic 
   cohomology class $[\theta_\Ucov]\in H^1(\Gcov/\Ncov,(\LG/\LN)^*)$ which measures the obstruction to making $\U$ a hamiltonian $\G$-space\textup, and vanishes if $\smash{\a_{|\LK}}=0$. If $\smash{\a_{|\LK}}\ne0$\textup, the derived class $[\<D\theta_\Ucov(e)(\cdot),\cdot\>]\in H^2(\LG/\LN,\RR)$ is that of the central extension
   \begin{equation}
      \label{extension}
      0
      \longrightarrow
      \LK/\LJ
      \longrightarrow
      \LL/\LJ
      \longrightarrow
      \LL/\LK
      \longrightarrow
      0
   \end{equation}
   where $\LJ=\ker(\smash{\a_{|\LK}})$.
\end{coro2}
    
\begin{proof}
   These are all elementary properties of the function $\theta_\Ucov$ in \eqref{thetaU}. First, by an easy application of (\ref{N_fixes_ann_n}ii), it satisfies the cocycle identity $\theta_\Ucov(\gcov\gcov') = \theta_\Ucov(\gcov)+\g(\theta_\Ucov(\gcov'))$ which defines $Z^1(\Gcov,\LG^*)$ \cite[p.\,112]{Bourbaki:1980}. As it takes its values in $\ann_{\LG^*}(\LN)$ \eqref{cocycle} and clearly only depends on $\gcov=\bracl\n,\l\bracr$ via its class $\bracl\N,\l\bracr$, we may equally regard it as a member of $Z^1(\Gcov/\Ncov,(\LG/\LN)^*)$. Moreover, replacing $\acheck$ by $\acheck+z$, also projecting on $\a\in\LN^*$, alters the cocycle by a term $z-\l(z)$ which is the coboundary of $-z\in\ann_{\LG^*}(\LN)$. So the resulting cohomology class $[\theta_\Ucov]$ does not depend on our choice of $\acheck$. Next, if $\smash{\a_{|\LK}}=0$ then we may take $\acheck$ to be $\a$ on $\LN$ and $0$ on $\LL$, in which case the cocycle vanishes:
   % \begin{equation}
      $
      \<\acheck-\l(\acheck),\LN+\LL\>
      =\<c-\l(c),\LN\>+\<0-\l(0),\LL\>
      =0
      $
   % \end{equation}
   since $\l$ stabilizes $\a$. Finally we consider the derivative $f_{\,\Ucov}=\<D\theta_\Ucov(e)(\cdot),\cdot\>$:
   \begin{align}
      f_{\,\Ucov}(\ln+\ll,\ln'+\ll')
      &=
      \tfrac d{dt}\left.
      \<\theta_\Ucov(\bracl\e{t\ln},\e{t\ll}\bracr),\ln'+\ll'\>
      \right|_{t=0}
      \notag
      \\
      &=
      \tfrac d{dt}\left.
      \<\acheck-\e{t\ll}(\acheck),\ln'+\ll'\>
      \right|_{t=0}
      \notag
      \\
      &=\<-\ll(\acheck),\ll'\>
      &
      &\text{since $\<\LL(\a),\LN\>=0$}
      % &\text{since $\ll$ stabilizes $\a$}
      \notag
      \\
      \label{fUcov}
      &=\<\acheck,[\ll,\ll']\>.
   \end{align}
   As above one verifies that this defines a cocycle in $Z^2(\LG/\LN,\RR)$ \cite[p.\,196]{Bourbaki:1980} whose class in $H^2(\LG/\LN,\RR)=H^2(\LL/\LK,\RR)$ \eqref{isom1} does not depend on the choice of $\acheck$. To see that it coincides with the cohomology class of the extension \eqref{extension}, we recall that the latter is that of the cocycle \cite[p.\,198]{Bourbaki:1980} 
   \begin{align}
       f(\ll+\LK,\ll'+\LK)
       &=
       \<\a,[s(\ll+\LK),s(\ll'+\LK)] - s([\ll+\LK,\ll'+\LK])\>
       % \notag
   \intertext{where $s$ is any linear section of the projection $\LL/\LJ\to\LL/\LK$, and where we have identified $\LK/\LJ$ with $\RR$ by sending $\varkappa+\LJ$ to $\<\a,\varkappa\>$. Choosing $s(\ll+\LK)=q(\ll)+\LJ$ where $p$ (resp.~$q$) is the projection onto $\LK$ (resp.~$\Lie v$) associated to a vector space decomposition $\LL=\LK\oplus\Lie v$, this is}
      &=
      \<\a,[q(\ll),q(\ll')] - q([\ll,\ll'])\>
      \notag
      \\
      &=
      \<\a,[\ll-p(\ll),\ll'-p(\ll')] - [\ll,\ll']+p([\ll,\ll'])\>
      \notag
      \\
      \label{fext}
      &=
      \<\a\circ p,[\ll,\ll']\>
   \end{align}
   since $\<\LL(\a),\LN\>=0$. Now $\a\circ p$ defines an extension of $\a$ to $\LL$ and thereby to $\LG=\LN+\LL$, which we can assume coincides with $\acheck$ (by changing the latter if necessary). So \eqref{fUcov} and \eqref{fext} belong to the same cohomo\-logy class, as was to be shown.
\end{proof}

\begin{defi}
   \label{obstruction}
   We shall refer to the class $[\theta_\Ucov]$ in \eqref{obstCoro} as the \emph{\textbf{symplectic Mackey obstruction}} of $\U$, and to the class $[\<D\theta_\Ucov(e)(\cdot),\cdot\>]$, or the central extension \eqref{extension} it represents, as the \emph{\textbf{infinitesimal Mackey obstruction}} of $\U$.
\end{defi}
	    
Given $\Theta\in H^1(\G,\LG^*)$ let us also agree to call \emph{\textbf{hamiltonian $(G,\Theta)$\nobreakdash-space}} a symplectic manifold $(\X,\w)$ with an $\w$-preserving action of $\G$ which admits a \moment map $\phi$ satisfying $\phi(g(x))=g(\phi(x))+\theta(g)$ identically for some $\theta\in\Theta$. (We emphasize that $\phi$ is not part of the structure: only its existence is.) With this in hand we can rephrase our solution \eqref{extTheo} of the extension problem in a form more closely parallel to the representation-theoretic version (\ref{step3}), namely:

\begin{coro3}
   Let $\N$ be normal in $\G$ and let $\U\in(\LN^*/\N)^\G$ have symplectic Mackey obstruction $\Theta$. Then the correspondence $\X\rightleftarrows\Z$ of Theorem \textup{\ref{extTheo}} can be regarded as a correspondence
   \begin{equation}
      \begin{gathered}
         \textit{$\N$-primary hamiltonian}\\[-.55ex]
         \textit{$\G$-spaces $\X$ over $\U$}
      \end{gathered}
      \qquad\rightleftarrows\qquad
      \begin{gathered}
         \textit{arbitrary hamiltonian}\\[-.55ex]
         \textit{$(\Gcov/\Ncov, -\Theta)$-spaces $\Z$.}
      \end{gathered}
   \end{equation}
\end{coro3}

\begin{proof}
   Let $(\X,\w,\PHI)$ be $\N$-primary over $\U$, and consider the corresponding $\K$-primary $\L$-space $(\Z,\wZ,\PSI)$ given by Theorem \eqref{extTheo}. We know \eqref{Ko(v)=0} that $\Ko$ acts trivially so that $\Z$ is in fact a $\L/\Ko=\Gcov/\Ncov$-space \eqref{isom1}. Moreover we have seen that this action admits a \moment map $\phiZ:\Z\to(\LG/\LN)^*$ \eqref{momentGcovZ} whose non-equivariance cocycle \eqref{theta} has cohomology class $-\Theta$.
   
   Conversely, let $(\Z,\wZ)$ be any hamiltonian $(\Gcov/\Ncov, -\Theta)$-space. By \eqref{isom1} we can regard it as a hamiltonian $(\L/\Ko, -\Theta)$-space, $\Theta\in H^1(\L/\Ko, (\LL/\LK)^*)$ being the class of the cocycle $\theta_\Ucov(\l\Ko)=\smash{(\acheck-\l(\acheck))_{|\LL}}\in\ann_{\LL^*}(\LK)$ (\ref{N_fixes_ann_n}iii, \ref{thetaU}, \ref{cocycle}). By definition this means that we can choose a \moment map $\phiZ:\Z\to\ann_{\LL^*}(\LK)$ so that 
\begin{equation}
   \phiZ(\l(\z))=\l(\phiZ(\z))+(\l(\acheck)-\acheck)_{|\LL}
\end{equation}
(equality in $\ann_{\LL^*}(\LK)=(\LL/\LK)^*$). Now, lifting the action from $\L/\Ko$ to $\L$ enables us to add $\acheck_{|\LL}$ to both sides to obtain the $\L$-equivariance of $\PSI:=\phiZ(\cdot)+\smash{\acheck_{|\LL}}$. Thus $(\Z,\wZ,\PSI)$ is a $\K$-primary hamiltonian $\L$-space over $\smash{\acheck_{|\LK}}$. By \eqref{extTheo} this corresponds to an $\N$-primary hamiltonian $\G$-space over $\U$, as desired.
\end{proof}

\begin{rema}
	Another justification for Definition \eqref{obstruction} is that it appears to correctly predict, at least at the Lie algebra level, the Mackey obstruction (\ref{step3}) of the representation(s) obtained by ``quantizing'' $\U$. To be specific, let us first note that \eqref{extension} does not always integrate to a Lie group extension. (Indeed it is not hard to construct examples where the subgroup $J\o$ generated by $\LJ$ in \eqref{extension} is not even closed.) However, when $\U$ is \emph{integral} in the sense that $\K$ admits a character $\chi$ with differential $i\smash{c_{|\LK}}$, then $J=\ker(\chi)$ is closed and the extension
	\begin{equation}
	   1
	   \longrightarrow
	   \K/J
	   \longrightarrow
	   \L/J
	   \longrightarrow
	   \L/\K
	   \longrightarrow
	   1
	\end{equation}
	integrates \eqref{extension} and is precisely the Mackey obstruction found in \cites[Thm IV.4.1]{Auslander:1971a}[§6]{Duflo:1970}, in the case where $\N$ is nilpotent and $\G$ is solvable.
\end{rema}

%%%%%%%%%%%%%%%%%%%%%%%%%%%%%%%%%%%%%%%%%%%%%%%%%%%%%%%%%%%%%%%%%%%%%%%%
\section{Examples}
%%%%%%%%%%%%%%%%%%%%%%%%%%%%%%%%%%%%%%%%%%%%%%%%%%%%%%%%%%%%%%%%%%%%%%%%
\begin{exemSSD}
%%%%%%%%%%%%%%%%%%%%%%%%%%%%%%%%%%%%%%%%%%%%%%%%%%%%%%%%%%%%%%%%%%%%%%%%

We consider here the group $\G$ of all block upper triangular matrices of the form
\begin{equation}
  g=
  \begin{pmatrix}
  1&\overline{\bm b}A&\frac12\|\bm b\|^2&f\\
  &A&\bm b&\bm c\\
  & & 1&e\\
  & & &1
  \end{pmatrix}
\end{equation}
where $A \in \mathrm{SO}(3)$, $\bm b,\bm c \in \RR^3$, $e,f \in \RR$ and the bar means transpose. The quotient obtained by forgetting the first row and column is the Galilei group and $G$ is (up to a $\ZZ_2$ cover which need not concern us here) its universal central extension, where $f$ is the central parameter. The axioms of classical mechanics as spelled out in \cite[{}13.1, 13.7]{Souriau:1970}, state: ``A free dynamical system is represented by a connected symplectic manifold $(\X,\w)$ together with a  $\w$-preserving action of the Galilei group admitting a \moment map''. As one knows \cite[p.\,66]{Kostant:1968}, the resulting action of $\G$ always admits an equivariant \moment map, so we can rephrase this as: \emph{A free dynamical system is represented by a connected hamiltonian $\G$-space where the center acts trivially}. Identifying $\LG^*$ with $\RR^{11}$ by writing $(\bm L, \bm G, \bm P, E, M)$ for the value of the 1-form
\begin{equation}
   \tfrac12\operatorname{Tr}\left(\left(\begin{smallmatrix}
   \phantom-0\phantom{_1}&\phantom-L_3&-L_2\\ 
   -L_3&\phantom-0\phantom{_1}&\phantom-L_1\\ 
   \phantom-L_2&-L_1&\phantom-0\phantom{_1}
   \end{smallmatrix}\right)
   dA\right) - \<\bm G,d\bm b\> + \<\bm P, d\bm c\> - Ede-Mdf
\end{equation}
at the identity, we observe that this implies, first, that $\X$ is primary over a point $\{M\}$ (known as the \emph{total mass} of the system), and secondly, when $M\ne0$, that $\X$ is also primary over a coadjoint orbit $\RR^6$ of the normal Heisenberg subgroup $\N\subset\G$ defined by $A=\mathrm{id}$, $e=0$. Applying the results \eqref{decTheo}, (\ref{trivCoro}iii) and \eqref{KoenigTheo}, we obtain after computing the \moment map $\phiU$:

1) $\X$ is symplectomorphic to the product of $\U=(\RR^6, Md\bm V\wedge d\bm R)$ by a symplectic manifold $\Z$ on which $\N$ acts trivially. This is the original barycentric decomposition theorem, first proved in this generality in \cite[Thm 13.15]{Souriau:1970}, but which in essence goes back at least to Euler \cite[p.\,187]{Euler:1750a}:
\begin{quote}
 % \small
	\emph{``Quelque composé que soit le mouvement $[\x]$ d'un corps solide, on le peut toujours décomposer en un mouvement progressif $[\u]$ et en un mouvement de rotation $[\z]$. Le premier s'estime par le mouvement du centre de gravité du corps, et il est toujours permis de con\-si\-dè\-rer ce mouvement séparément et indépendamment de l'autre mouvement.''}
\end{quote}

2) The group $\G\times\mathrm{SO}(3)\times\RR$ acts symplectically on $\X=\U\times\Z$ and the \moment map $\PHI:\X\to\LG^*$ splits as
% \begin{equation}
   $
   \PHI(u,v)
   =
   M\bigl(\bm R\times \bm V, \bm R, \bm V, \smash{\tfrac12}\|\bm V\|^2, 1\bigr)
   +
   \bigl(\bm L', 0, 0, E', 0\vphantom{\tfrac12}\bigr)
   $
% \end{equation}
where $\bm L'$ and $E'$, known as the \emph{proper} angular momentum and energy,
depend only on $\z$. This is the classical Kœnig theorem \cite[{}13.35]{Souriau:1970}, thus named by Painlevé \cite{Painleve:1930} but going back to Laplace for $\bm L$ \cite [p.\,70]{Laplace:1799} and to Kœnig for $E$ \cite[p.\,173]{Koenig:1751}: 
\begin{quote}
 % \small
	\emph{``Vis viva $[2E]$, quæ inest corporibus A et B, æqualis est Vi viva communi $[M\|\bm V\|^2]$, una cum Viribus vivis propriis amborum corporum $[2E']$.''}
\end{quote}
\end{exemSSD}

\begin{rema}
   \cite{Souriau:1971a} gives another application of the case where $\N$ is Heisenberg, to the description of a particle in a constant electromagnetic field.

\end{rema}

%%%%%%%%%%%%%%%%%%%%%%%%%%%%%%%%%%%%%%%%%%%%%%%%%%%%%%%%%%%%%%%%%%%%%%%%
\begin{exemNS}
%%%%%%%%%%%%%%%%%%%%%%%%%%%%%%%%%%%%%%%%%%%%%%%%%%%%%%%%%%%%%%%%%%%%%%%%
Let $\G$ be the solvable group of all upper triangular matrices of the form
\begin{equation}
   \label{nonsplitG}
   \g=
   \begin{pmatrix}
   1&c&0       &e&f\\
   &1&0       &0&e\\
   & &\e{\j a}&0&b\\
   & &        &1&a\\
   & &        & &1
   \end{pmatrix}
   \qquad
   \begin{aligned}
     a,c,e,f&\in\RR\\
     b&\in\CC\\
     \j&=2\pi\smash{\text{\small$\sqrt{-1}$}},
   \end{aligned}
\end{equation}
and let $\N$ be the normal subgroup in which $c=e=0$. We identify $\LG^*$ with $\RR\times\CC\times\RR^3$ by writing $(p,z,r,s,t)$ for the value  of the 1-form 
\begin{equation}
	pda+\Re(\bar zdb)-rdc-sde-tdf
\end{equation}
  at the identity; likewise we identify $\LN^*$ with triples $(p,z,t)$ so that the projection $\LG^*\to\LN^*$ writes $(p,z,r,s,t)\mapsto(p,z,t)$. The coadjoint action of $\G$ preserves the hyperplane $t=1$ and writes there
\begin{equation}
   \label{Ad*(g)}
   g\begin{pmatrix}
      p\\
      z\\
      r\\
      s\\
      1   
   \end{pmatrix}
   =
   \begin{pmatrix}
      p+e+\Re\bigl(\overline{\j b}\,\e{\j a}z\bigr)\\
      \e{\j a}z\\
      r+e\\
      s+a-c\\
      1  
   \end{pmatrix}
\end{equation}
in $\LG^*$, and likewise in $\LN^*$ with the rows for $r,s$ erased. Taking $\acheck=(0,1,0,0,1)$, it is clear from this that the coadjoint orbit $\X=\G(\acheck)$ is $\N$-primary over $\U$ with fiber $\Z$, where
\begin{align}
	\X
	&=\bigl\{\bigl(p,\e{\j q},r,s,1\bigr):p,q,r,s\in\RR\bigr\},\\
	\label{nonsplitU}
	\U
	&=\bigl\{\bigl(p,\e{\j q},1\bigr):p,q\in\RR\bigr\}, 
	&&\hspace{-3.05em}\w_\U=dp\wedge dq,\\
   \Z
   &=\bigl\{\bigl(0,1,r\vphantom{^{\j q}},s,1\bigr):r,s\in\RR\bigr\}, 
   &&\hspace{-3em}\w_\Z=dr\wedge ds.
\end{align}
We claim that despite first appearances, $\X$ splits neither as $\U\times\Z$ nor in fact in any other way. A first hint of this is the observation that its 2-form works out as $\w_\X=dp\wedge dq+dq\wedge dr+ dr\wedge ds$ which is not the sum of $\w_\U$ and $\w_\Z$. For an actual proof we note that since the fiber $\Z$ is connected it is enough to see that $\EFG$, or equivalently the stabilizer~$\K$, acts nontrivially on $\Z$ \eqref{nontrivCoro}. But this is clear since one finds
\begin{equation}
  \K=\left\{\k=
  \left(
  \begin{smallmatrix}
  1&0&0       &0&f\\
   &1&0       &0&0\\
   & &1&0&b\\
   & &        &1&a\\
   & &        & &1
  \end{smallmatrix}
  \right):
  \quad
  \begin{aligned}
     a&\in\ZZ\\
     b,f&\in\RR
  \end{aligned}
  \right\}
\end{equation}
(hence $\EFG=\ZZ$) and, from \eqref{Ad*(g)}, $k(0,1,r,s,1)=(0,1,r,s+a,1)$.
\end{exemNS}

%%%%%%%%%%%%%%%%%%%%%%%%%%%%%%%%%%%%%%%%%%%%%%%%%%%%%%%%%%%%%%%%%%%%%%%%
\begin{remas}\label{notPrimary}
%%%%%%%%%%%%%%%%%%%%%%%%%%%%%%%%%%%%%%%%%%%%%%%%%%%%%%%%%%%%%%%%%%%%%%%%
(i) As the reader may verify, this orbit also provides an example of nonvanishing Mackey obstruction (\ref{obstruction}), viz.~the cohomology class~of
\begin{equation}
   \theta_\Ucov(g,l_{\,45})=(0,0,-e,c-l_{\,45},0).
\end{equation}
Here we have identified the group $\Gcov$ \eqref{Gcov} with the direct product $\G\times\ZZ$ by sending $\bracl n,l\bracr$ to the pair $(nl, \text{entry $l_{\,45}$ of }l)$.

(ii) It is of interest to note that the representations attached to $\X$ are \emph{not} primary when restricted to $\N$. In fact, \cite[Prop.~3.8]{Ziegler:1996a} applied to the normal subgroup $P\o$ in which $a=e=0$ shows that we have $\X=\Ind_P^\G\{\smash{\acheck_{|\Lie p}}\}$ where $P$ is the subgroup in which $a\in\ZZ$, $e=0$. Accordingly one finds that the representations attached to $\X$ by \cite{Auslander:1971a} are the $\Ind_P^\G\,\chi_w$ where
\begin{equation}
   \chi_w\left(\begin{smallmatrix}
   1&c&0       &0&f\\
   &1&0       &0&0\\
   & &1&0&b\\
   & &        &1&a\\
   & &        & &1
   \end{smallmatrix}
   \right)=w^a\e{i\Re(b)}\e{-if},
   \qquad
   w=\e{i\theta}\in S^1,
\end{equation}
is the most general character of $P$ with differential $i\smash{\acheck_{|\Lie p}}$. These representations can be realized in $L^2(G/P)=L^2(S^1\times\RR)$ by
\begin{equation}
   \label{repG}
   (gE)(z,r)=
   \e{i(\theta+r)a}
   \e{i\Re(\bar z b)}
   \e{-ic(r-e)}
   \e{-if}
   E(z\e{-\j a}, r-e).
\end{equation}
Likewise one finds that the representations attached to $\U$ are the $\Ind_{N\cap P}^N\,\eta_w$ where $\eta_w$ is the restriction of $\chi_w$, which we can realize in $L^2(S^1)$ by
\begin{equation}
   \label{repN}
   (nF)(z)=
   \e{i\theta a}
   \e{i\Re(\bar z b)}
   \e{-if}
   F(z\e{-\j a}).
\end{equation}
Now it is clear that on restricting \eqref{repG} to $\N=\{\g: c=e=0\}$ one gets not one of the representations \eqref{repN} but a direct integral of all of them.
\end{remas}

%%%%%%%%%%%%%%%%%%%%%%%%%%%%%%%%%%%%%%%%%%%%%%%%%%%%%%%%%%%%%%%%%%%%%%%%
\begin{exemNTS}
%%%%%%%%%%%%%%%%%%%%%%%%%%%%%%%%%%%%%%%%%%%%%%%%%%%%%%%%%%%%%%%%%%%%%%%%
Let $H$ be the subgroup of \eqref{nonsplitG} in which $c=0$. Its coadjoint action is obtained by making $c=0$ and erasing the row for $r$ in \eqref{Ad*(g)}. In particular we see that the orbit
\begin{equation}
	Y=H(0,1,0,1)=\bigl\{\bigl(p,\e{\j q},q,1\bigr):p,q\in\RR\bigr\}
\end{equation}
is $\N$-primary over the same orbit \eqref{nonsplitU} of the same normal subgroup $\N$, and is its universal covering.
\end{exemNTS}

%%%%%%%%%%%%%%%%%%%%%%%%%%%%%%%%%%%%%%%%%%%%%%%%%%%%%%%%%%%%%%%%%%%%%%%%
\begin{rema}
%%%%%%%%%%%%%%%%%%%%%%%%%%%%%%%%%%%%%%%%%%%%%%%%%%%%%%%%%%%%%%%%%%%%%%%%
As in (\ref{notPrimary}ii) one can verify that the representations attached to $Y$ are not primary when restricted to $\N$. This raises the question of a possible correlation between an irreducible representation being $\N$-primary and the corresponding orbit being \emph{trivially split} $\N$-primary \eqref{splitDefi}.
\end{rema}

%%%%%%%%%%%%%%%%%%%%%%%%%%%%%%%%%%%%%%%%%%%%%%%%%%%%%%%%%%%%%%%%%%%%%%%%
\begin{exemTNS}
   \label{Kodaira-Thurston}
%%%%%%%%%%%%%%%%%%%%%%%%%%%%%%%%%%%%%%%%%%%%%%%%%%%%%%%%%%%%%%%%%%%%%%%%
All our examples so far were at least \emph{homeomorphic} to the product of a covering of $\U$ with another manifold, so it is natural to ask how far from such a product a primary space can get. In this section we exploit the close similarity of \eqref{prodFib} with the Kodaira-Thurston construction \cites{Thurston:1976}[p.\,10]{Weinstein:1977a} to show that, at least as long as we don't insist on homogeneity under an ambient Lie group $\G$ acting in hamiltonian fashion, examples exist with no such homeomorphism.

To this end, we let $\N$ and $\U$ be same group and cylinder orbit as in \eqref{nonsplitU}, so that $\Ucov$ and $\EFG$ are respectively the plane $\RR^2$ and the integers $\ZZ$. We take for $\Z$ the flat torus $\TT^2=\{(\e{\j\r},\e{\j\s}):\r,\s\in\RR\}$ where the integers act by powers of a Dehn twist: $k(R,S)=(RS^k,S)$, and we form the associated bundle~$\X=\Ucov\times_\EFG\Z$. Thus $\X$ is the quotient of $(\RR^2\times\TT^2,dp\wedge dq+dr\wedge ds)$ by the $\ZZ$-action
\begin{equation}
   \label{Dehn}
   % k(p , q , \r+\ZZ, \s+\ZZ)=(p , q +k , \r+ k\s +\ZZ, \s+\ZZ).
   k\bigl(p,q ,\e{\j\r},\e{\j\s}\bigr)
   =\bigl(p,q+k,\e{\j(\r+k\s)},\e{\j\s}\bigr).
\end{equation}
Our claim is that $\X$ is not homeomorphic to the product of any covering of $\U$ by any surface. To see this, we observe that $\RR^2\times\TT^2$ can be regarded as the right quotient of the group $M$ of real matrices
\begin{equation}
  \label{nilmanifold}
m = 
\begin{pmatrix}
1 & 0 & 0 & p \\
  & 1 &\r &\s \\ 
  &   & 1 & q \\ 
  &   &   & 1
\end{pmatrix}
\mbox{ by the subgroup } \ZZ \oplus \ZZ = 
\begin{pmatrix}
1 & 0 & 0 & 0\\
  & 1 & \ZZ & \ZZ\\
  &   & 1 & 0\\
  &   &   & 1
\end{pmatrix}.
\end{equation}
Moreover \eqref{Dehn} coincides with the residual right action on $M/(\ZZ\oplus\ZZ)$ of the 
% discrete 
Heisenberg group
\begin{equation}
   H_\ZZ=
   \begin{pmatrix}
   1 & 0 & 0 & 0\\
     & 1 & \ZZ & \ZZ\\
     &   & 1 & \ZZ\\
     &   &   & 1
   \end{pmatrix}
\end{equation}
% $H_\ZZ=\{g: p=0 \text{ and } q,\r,\s\in\ZZ\}$ 
which normalizes $\ZZ\oplus\ZZ$. Thus we have $\X=M/H_\ZZ$ and therefore $\pi_1(\X)=H_\ZZ$. Now since $H_\ZZ$ does not admit $\ZZ$ as a direct factor, it follows that $\X$ cannot be the product of the cylinder $\U$ (or any finite covering) by anything. Nor is it the product of the plane $\Ucov$ by any surface, because no surface has fundamental group $H_\ZZ$ \cite[pp.\,135, 143]{Massey:1977}.
\end{exemTNS} 
 
% \input variantes.english.tex
% \input variantes.french.tex
% \input tests.tex

%%%%%%%%%%%%%%%%%%%%%%%%%%%%%%%%%%%%%%%%%%%%%%%%%%%%%%%%%%%%%%%%%%%%%%%%
% \bibliographystyle{amsalpha}
% \bibliography{primary}
\setlength{\labelalphawidth}{3.2em}
% \setlength{\labelalphawidth}{40pt}

%%% restore this:
\let\l\polishl

\printbibliography
%%%%%%%%%%%%%%%%%%%%%%%%%%%%%%%%%%%%%%%%%%%%%%%%%%%%%%%%%%%%%%%%%%%%%%%%

\end{document}